\documentclass[11pt,reqno,UTF8,twoside]{amsart}
\usepackage{mathrsfs}
\usepackage{amsfonts,amssymb,amsmath,amsthm}
\usepackage[noadjust]{cite}
\usepackage[colorlinks=true,citecolor=red,linkcolor=blue]{hyperref}
\usepackage{tocvsec2}
\usepackage{titletoc}
\usepackage{geometry}
\usepackage{bm}
\usepackage{indentfirst}
\usepackage{graphicx}
\usepackage{float} 
\usepackage{booktabs}
\usepackage{longtable}
\usepackage{stmaryrd}
\usepackage{enumerate}
\usepackage{ulem}
\usepackage{color,soul}
\usepackage{setspace}
\usepackage[pagewise]{lineno} 

\usepackage{appendix}
\usepackage{cases}
\usepackage{todonotes}
\usepackage{enumitem}

\numberwithin{equation}{section}
\newtheorem{theorem}{Theorem}[section]
\newtheorem{lemma}[theorem]{Lemma}
\newtheorem{corollary}[theorem]{Corollary}
\newtheorem{proposition}[theorem]{Proposition}
\newtheorem{remark}[theorem]{Remark}
\newtheorem{definition}[theorem]{Definition}

\newtheorem*{thm}{Main Theorem}

\newcommand{\T}{\mathbb{T}}
\newcommand{\N}{\mathbb{N}}
\newcommand{\R}{\mathbb{R}}
\newcommand{\Z}{\mathbb{Z}}

\newcommand{\Pb}{\mathbb{P}}
\newcommand{\supp}{\operatorname{supp}}

\newcommand{\dist}{\operatorname{dist}}
\newcommand{\Lip}{\operatorname{Lip}}

\makeatletter
\@namedef{subjclassname@2020}{\textup{2020} Mathematics Subject Classification}
\makeatother

\allowdisplaybreaks

\begin{document}
	
\title[Mixing of random KdV]{\large E{\MakeLowercase{xponential mixing for} K\MakeLowercase{orteweg--de} V\MakeLowercase{ries equation with localized noise}}}

\author[Y. Chen, S. Xiang, Z. Zhang, J.-C. Zhao]{Y\MakeLowercase{uxuan} C\MakeLowercase{hen}, \; S\MakeLowercase{hengquan} X\MakeLowercase{iang}, \;  Z\MakeLowercase{hifei} Z\MakeLowercase{hang}, \; J\MakeLowercase{ia-}C\MakeLowercase{heng} Z\MakeLowercase{hao}}

\address[Yuxuan Chen]{School of Mathematical Sciences, Peking University, 100871, Beijing, China.}
\email{chen\underline{ }yuxuan@pku.edu.cn}

\address[Shengquan Xiang]{School of Mathematical Sciences, Peking University, 100871, Beijing, China.}
\email{shengquan.xiang@math.pku.edu.cn}

\address[Zhifei Zhang]{School of Mathematical Sciences, Peking University, 100871, Beijing, China.}
\email{zfzhang@math.pku.edu.cn}

\address[Jia-Cheng Zhao]{School of Mathematical Sciences, Shenzhen University, 518061, Shenzhen, China.}
\email{zjc@szu.edu.cn}



\thanks{Shengquan Xiang is partially  supported by NSFC 12571474. Zhifei Zhang is partially supported by NSFC 12288101.}
	
\begin{abstract}
We establish exponential mixing for the randomly forced and weakly damped KdV equation in $L^2(\T)$.
The noise is bounded, localized, and degenerate in high frequencies. Our proof relies on a general probabilistic framework in \cite{LWX-24,CXZZ-25}, nonlinear smoothing for KdV and its linearization via normal form transformation, and stabilization of the system by localized force. This paper continues a series of works connecting asymptotic compactness, control theory, and ergodicity and mixing for randomly forced dispersive PDEs.
\end{abstract}
	
\maketitle

\settocdepth{section}
\titlecontents{subsection}[5em]{}{\contentslabel{2em}}{\hspace*{-2em}}{\hfill\contentspage}

\tableofcontents

\section{Introduction}\label{Section-Introduction}

We investigate asymptotic dynamics of the KdV equation on the torus $\T=\R/2\pi\Z$, in the presence of weak damping and random forcing, which reads
\begin{equation}\label{Problem-random}
    \begin{cases}
        \partial_t u+\partial_x^3 u+a u+\tfrac{1}{2}\partial_x (u^2)\  =\eta,\\
        u(0,\cdot)=u_0\in \dot{L}^2(\T).
    \end{cases}
\end{equation}
Here, $a>0$ is the damping coefficient, $\dot{L}^2(\T)$ denotes the space of mean-zero $L^2$ functions, and $\eta=\eta(t,x)$ represents a random noise {\it localized} in a space-time region. Our main result establishes {\it exponential mixing} for the associated Markov semigroup: the law of solution converges to the unique invariant measure at an exponential rate.

The KdV equation (with $a=0$ and $\eta\equiv 0$) holds significant mathematical and physical importance, originally modeling unidirectional wave propagation in shallow water. In addition to being a dispersive PDE, it is also Hamiltonian and completely integrable. We refer the reader to \cite{Bourgain-93,CKSTT-03,KT-06,BIT-11,KV-19} and references therein for the fruitful well-posedness results. 

Our interest in weak damping and localized noise in \eqref{Problem-random} stems from the recent works \cite{LWX-24,CLXZ-24,CXZZ-25,CX-26} on the statistics of random hyperbolic/dispersive equations. The weak damping dissipates energy without simultaneously regularizing the solution (this contrasts with strong damping represented by negative Laplacian operator \cite{KP-08,Kuk-10}). In the deterministic setting, it leads to the existence of compact attractor \cite{Gou-18}.

The localized feature of noise was introduced to the study of mixing by Shirikyan \cite{Shi-15} for 2D Navier--Stokes system. We refer to, e.g.~\cite{Shi-21,Ner-24,LWX-24,CXZZ-25} for further contributions to parabolic and dispersive equations with localized noise. A key mathematical motivation for studying localized noise lies in the deep connection between long-time behavior and control properties.

\subsection{Main results}\label{Section-mainresult}

Throughout this paper, $(t_1,t_2)\subset (0,1)$ and $(s_1,s_2)\subset\T$ are fixed temporal and spatial intervals, respectively (we parametrize $\T$ as $[0,2\pi]$ modulo $0\sim 2\pi$). Let
\begin{equation}\label{eq-chi}
\chi(t,x)=\mathbf{1}_{(t_1,t_2)}(t)\mathbf{1}_{(s_1,s_2)}(x)
\end{equation}
be the characteristic function of $(t_1,t_2)\times (s_1,s_2)$. We introduce the non-constant eigenfunctions of Laplacian operator on $(s_1,s_2)$ with Neumann boundary conditions,
\begin{equation}\label{orthonormal-basis}
\beta_k(x)=\sqrt{\frac{2}{s_2-s_1}}\cos \left(k\pi \frac{x-s_1}{s_2-s_1}\right),\quad k\in \N^+.
\end{equation}
Then $\{\beta_k;k\in \N^+\}$ forms an orthonormal basis of mean-zero $L^2(s_1,s_2)$ functions. We also fix an orthonormal basis $\{\alpha_j;j\in \N^+\}$ of $L^2(t_1,t_2)$. Extending by zero, we regard $\alpha_j\in L^2(0,1)$ and $\beta_k\in \dot{L}^2(\T)$. The smoothness of $\beta_k$ on $[s_1,s_2]$ yields $\beta_k\in \dot{H}^{1/2-\varepsilon}(\T)$ for any $\varepsilon>0$.

\medskip

\noindent{\bf Setting of noise.} {\it The law of $\eta(t,x)$ is statistically $1$-periodic, in the sense that
\begin{equation*}
\eta_n(t,\cdot):=\eta(t+n,\cdot),\quad t\in [0,1),\,n\in\N
\end{equation*}
is a sequence of i.i.d.~random variables in $L^2(0,1;\dot{L}^2(\T))$. Moreover, $\eta_n$ can be specified as
\begin{equation}\label{Noise-structure-1}
\eta_n(t,x)=\chi(t,x)\sum_{j,k\in \N^+}b_{j,k}\theta^n_{j,k}\alpha_j(t)\beta_k(x).
\end{equation}
Here, $b_{j,k}\ge 0$ are non-negative constants decaying sufficiently fast (see condition \eqref{noisestrength} below), and $\theta^n_{j,k}$ are independent scalar random variables. In addition, the law of $\theta^n_{j,k}$ admits a $C^1$ density function $\rho_{j,k}$ supported in $[-1,1]$, which satisfies $\rho_{j,k}(0)>0$.}

\medskip

Under this setting, the random KdV equation \eqref{Problem-random} admits a unique solution $u\in C(\R^+;\dot{L}^2(\T))$ almost surely. The periodicity of noise turns the discrete-time solutions
\[u_n:=u(n),\quad n\in \N\]
into a Markov process $(u_n,\Pb_{u_0})$ on $\dot{L}^2(\T)$. We write $\mathscr{D}(u_n)$ for the law of $u_n$, which belongs to the family of Borel probability measures $\mathcal{P}(\dot{L}^2(\T))$.

\begin{thm}\label{Main-theorem-1}
Given $\mathbf{B}>0$ and $\sigma\in (0,1/4)$, there exists an integer $N\in \N^+$ such that if, in addition to the above noise setting, the non-negative numbers $b_{j,k}$ in \eqref{Noise-structure-1} satisfy
\begin{equation}\label{noisestrength}
\sum_{j,k\in \N^+} b_{j,k}^2 \|\alpha_j\|_{L^2(0,1)}^2 \|\beta_k\|_{\dot{H}^{1/4+\sigma}(\T)}^2\leq \mathbf{B}^2,
\end{equation}
and
\begin{equation}\label{noisedegenerate}
b_{j,k}\not =0,\quad \forall 1\le j,k\le N,
\end{equation}
then the Markov process $(u_n,\Pb_{u_0})$ admits a unique invariant measure $\mu\in \mathcal{P}(\dot{L}^2(\T))$, which is supported in a bounded subset of $\dot{H}^{1/4+\sigma}(\T)$. Moreover, there are constants $C,\gamma>0$, such that for any initial data $u_0\in \dot{L}^2(\T)$, 
\begin{equation*}
\|\mathscr D(u_n)-\mu\|_{L}^*\leq C\left(1+\|u_0\|_{\dot{L}^2}^2\right)e^{-\gamma n},\quad\forall n\in\N,
\end{equation*}
where $\|\cdot\|_L^*$ denotes the dual-Lipschitz distance in $\dot{L}^2(\T)$ (defined in Section~{\rm\ref{Section-notation}} below).
\end{thm}

\begin{remark}
To the best of our knowledge, this is the first unique ergodicity result for the KdV equation in $L^2$. Recently, Glatt-Holtz, Martinez and Richards {\rm\cite{GHMR-25}} established unique ergodicity in $H^m\, (m\ge 2)$ for white-in-time noise, using a weak Foias--Prodi estimate in expectation. Besides overcoming the analytical challenges due to the lower $L^2$ regularity, another contribution of the present work is a control result along trajectories (Theorem~{\rm\ref{Theorem-control}}). Such pathwise information has potential for other relevant studies on random PDEs, for instance highly degenerate noise {\rm\cite{Ner-24,KNS-20,NZZ-24,LXZ-25,CXZ-26}}, and large deviation principle {\rm\cite{JNPS-15,CX-26}}.
\end{remark}

\subsection{Strategy of proof}\label{Section-strategy}

We employ a general mixing criterion from \cite{LWX-24,CXZZ-25}, which is based on coupling methods and asymptotic compactness. This criterion can be viewed as a dispersive counterpart of the approach by Kuksin and Shirikyan \cite{KS-12} and Shirikyan \cite{Shi-15} for parabolic systems, with asymptotic compactness playing a central role in compensating for the absence of smoothing effect. Three hypotheses of this criterion can be roughly summarized as follows:

\begin{itemize}
    \item[\tiny$\bullet$] (Exponential asymptotic compactness) The system \eqref{Problem-random} admits a bounded subset $\mathcal Y\subset \dot{H}^{1/4+\sigma}$ that attracts solutions exponentially in a pathwise manner:
    \[\dist_{\dot{L}^2}(u_n,\mathcal{Y})\le C\left(1+\|u_0\|_{\dot{L}^2}^2\right) e^{-\gamma n},\quad \forall u_0\in \dot{L}^2,\ n\in \N.\]

    \item[\tiny$\bullet$] (Irreducibility on $\mathcal Y$) Trajectories issued from $\mathcal Y$ can reach arbitrarily small neighborhoods of $0$: for any $r>0$, one can find $m\in \N^+$ such that
    \[\inf_{u_0\in \mathcal{Y}} \Pb_{u_0} (u_m\in B_{\dot{L}^2}(r))>0.\]

    \item[\tiny$\bullet$] (Coupling condition on $\mathcal Y$) For nearby initial states $u_0,w_0\in \mathcal Y$, we can construct two noise realizations (which is referred to as a coupling), such that with high probability,
    \[\|u_n-w_n\|_{\dot{L}^2}\le Ce^{-\gamma n},\quad \forall n\in \N.\]
\end{itemize}

\noindent We refer to Section~\ref{Section-criterion} for the precise statements.

Heuristically, the exponential asymptotic compactness reduces the problem to the attractor  $\mathcal{Y}$, the irreducibility implies a uniform positive probability for two trajectories from $\mathcal{Y}$ to become arbitrarily close, and the coupling condition indicates exponential convergence with high probability once they are sufficiently close at some time. If the coupling fails, we wait for irreducibility to bring trajectories close again, and then re-apply the coupling. Iterating this procedure with stopping time arguments then yields exponential mixing.

\medskip

As realized in \cite{CXZZ-25}, the verifications of both asymptotic compactness and coupling condition rely on {\it nonlinear smoothing}, which can be formulated as
\begin{equation}\label{eq_nonlinearsmoothing}
u(t)-e^{-(\partial_x^3+a)t} u_0\in \dot{H}^s(\T),\quad \forall u_0\in \dot{L}^2(\T)
\end{equation}
for some $s>0$. Note that both $u$ and the linear evolution $e^{-(\partial_x^3+a)t} u_0$ retain the $L^2$ regularity of the initial data, while their difference (namely the nonlinear part of the solution) gains extra regularity. Using multilinear estimates in Bourgain spaces and normal form transformation, we propose a variant  for linearized equation; see Proposition~\ref{Prop-linearized_nonlinearsmooothing}.

The nonlinear smoothing phenomenon was first discovered by Bourgain \cite{Bourgain-98} for nonlinear Schr\"odinger (NLS) equation, and has now become important to various topics in dispersive PDEs, including low-regularity well-posedness, growth of Sobolev norms, and global attractors. In the context of KdV equation, Babin, Ilyin and Titi \cite{BIT-11} and Colliander, Keel, Staffilani, Takaoka
and Tao \cite{CKSTT-03}, among others, explored analogous smoothing properties. Then the explicit formulation of nonlinear smoothing is established by Erdo\u{g}an and Tzirakis \cite{ET-13-kdv}. For KP-II equation (which can be viewed as a 2D generalization of KdV), nonlinear smoothing was proved by Isaza, Mej\'ia and Tzvetkov \cite{IMT-06}. Further results along this line can be found in, e.g.~\cite{OS-21,COS-24,ST-25}. 

\medskip

While asymptotic compactness is a consequence of nonlinear smoothing, and the irreducibility follows easily from the $L^2$ energy inequality, the coupling condition is more delicate and constitutes the core challenge. It is naturally related to ``stabilization along trajectory" as a control problem, which we now outline informally (see Section~\ref{Section-control} for rigorous discussions).

Since the noise $\eta_0$ in time $[0,1]$ is non-degenerate in the low frequencies by \eqref{noisedegenerate}, a perturbation $\eta_0+\chi \mathcal{P}_N \xi$ is small in total variation norm, where $\mathcal{P}_N$ is the orthogonal projection 
\begin{equation}\label{eq-PN}
    \mathcal{P}_N\colon L^2(t_1,t_2;\dot{L}^2(s_1,s_2))\to {\rm span} \{\alpha_j(t)\beta_k(x); 1\le j,k\le N\},
\end{equation}
and $\xi$ is a control function to be chosen. According to optimal coupling theory, we can construct another realization of the noise $\widetilde{\eta}_0$, with the same law as $\eta_0$, such that with high probability, 
\[\widetilde{\eta}_0=\eta_0+\chi \mathcal{P}_N \xi.\]

Using these two realizations of noise $\eta_0$ and $\tilde{\eta}_0$ for the dynamics of $w$ and $u$ on $[0,1]$, then the difference $u-w$ is governed, up to higher-order terms, by the linearized equation around $w$:
\begin{equation}\label{Problem-linearcontrol-0}
\partial_t v+\partial^3_xv+av+\partial_x(wv)=\chi\mathcal{P}_N\xi.
\end{equation}
Here $v(0)=v_0:=u_0-w_0$, and $v(1)\approx u(1)-w(1)$. If we can find a control $\xi$ so that 
$$
\|v(1)\|_{L^2}\le q\|v(0)\|_{L^2}
$$
for some $q\in (0,1)$, then the coupling condition follows by iteration on $[n,n+1]$ for each $n\in \N$.

To this end, we exploit the idea of frequency analysis as in \cite{LWX-24, CXZZ-25}, splitting the system into two parts in terms of low/high frequencies. Accordingly, the proof is divided into two parts:

\begin{itemize}
\item[\tiny$\bullet$] {\it Part 1 {\rm(}Low-frequency controllability}). The first $m$ spatial modes can be steered to $0$ by the finite-dimensional control $\chi \mathcal{P}_N \xi$; see Proposition~\ref{Prop-LFcontrol}.
This will be achieved by observability inequality for linear dual system of \eqref{Problem-linearcontrol-0}, combined with an adaptation of the Hilbert uniqueness method (HUM). Propagation properties from microlocal analysis, as well as unique continuation via Carleman estimate also come into play. These ideas are in light of the global stabilization by Laurent, Rosier and Zhang \cite{LRZ-10}.

\item[\tiny$\bullet$] {\it Part 2 {\rm(}High-frequency dissipation}). Owing to nonlinear smoothing \eqref{eq_nonlinearsmoothing} for linearized system (Proposition~\ref{Prop-linearized_nonlinearsmooothing}), the contribution of potential term $\partial_x (wv)$ belongs to a higher Sobolev space. Consequently, in high frequencies the damping term prevails, yielding a dissipation effect independent of the control. It is worth noting that the extra $H^{1/4+\sigma}$ regularity of $w$ (guaranteed by the attractor $\mathcal{Y}$) is crucial to nonlinear smoothing.
\end{itemize}

\begin{remark}
Owing to the specific dispersion relation of the KdV equation and the localized nature of the noise, the current paper differs from prior works {\rm\cite{LWX-24,CXZZ-25}} on wave and NLS equations in the following aspects:

\begin{itemize}
\item[$(1)$] The mechanism underlying nonlinear smoothing \eqref{eq_nonlinearsmoothing} is more intricate. Indeed, for wave equation, the argument is rather straightforward, thanks to Strichartz estimate and one derivative gain from the source term. And for NLS, one must remove the resonant terms, and then exploit a multilinear estimate that provides a regularity gain for the non-resonant terms. Nevertheless, for KdV equation, such bilinear estimate fails:
\[\|\partial_x (uv)\|_{X^{s+\varepsilon,-1/2}}\not\le C\|u\|_{X^{s,1/2}}\|v\|_{X^{s,1/2}}.\]
Inspired by {\rm\cite{ET-13-kdv}}, we perform the normal form reduction via integration by parts in time to deal with the nonlinearity $\partial_x(u^2)$ and the potential term $\partial_x(wv)$.

\item[$(2)$] In order to be compatible with the conservation of mass, we use the characteristic function $\chi$ instead of smooth cut-offs in the noise structure \eqref{Noise-structure-1}. This brings extra difficulties to the control problem, in comparison with previous studies {\rm\cite{RZ-96,RZ-06,LRZ-10}}, especially to the proof of observability inequality (see the discussion preceding the proof of Lemma~{\rm\ref{Lemma-fullobs}}).

\item[$(3)$] In most of the existing literature on mixing for dispersive PDEs, the regularity of solution is at least $H^1$, whereas our results are stated in $L^2$. This low regularity causes essential difficulties to PDE analysis. For example, in the adjoint system \eqref{Adjoint-problem}, the product $w\partial_x \varphi$ is not well-defined in Sobolev spaces for $w\in H^{1/4+\sigma}$ and $\varphi\in L^2$, and we need new bilinear estimates in Bourgain spaces (Lemma~{\rm\ref{Lemma-bilinear}} and Corollary~{\rm\ref{Coro-bilinear}}). Moreover, a unique continuation property (Proposition~{\rm\ref{Prop-UCP}}) becomes more subtle.
\end{itemize}
\end{remark}

\subsection{Prior works}\label{Section-priorwork}

We present a brief historical review on the mixing and control of KdV equation. As the literature is now extensive, we mention only the most relevant works.

\subsubsection{Mixing of randomly forced PDEs}

The study of ergodicity and mixing property of invariant measure for randomly forced parabolic PDEs has achieved rich developments, with the 2D Navier--Stokes equations 
\[\partial_t u-\Delta u+(u\cdot\nabla)u+\nabla p=\eta,\quad {\rm div}(u)=0\]
as a paradigmatic example in this field. Early results are surveyed in the monographs \cite{KS-12,Deb-13}. The seminal works \cite{HM-06,HM-08} of
Hairer and Mattingly established the first exponential mixing result when the noise is white in time and highly degenerate in Fourier modes, by introducing Malliavin calculus and hypoellipticity. Shirikyan \cite{Shi-15,Shi-21} proved exponential mixing for interior/boundary localized noise.  Kuksin, Nersesyan and Shirikyan \cite{KNS-20} treated highly degenerate Haar noise using geometric control methods. Notably, \cite{KNS-20,Shi-15,Shi-21} convey a link between mixing and controllability through coupling methods.

A natural question is whether similar results are available for hyperbolic/dispersive equations, where parabolic regularization is absent. For wave equations
\[u_{tt}-\Delta u+a \partial_t u+f(u)=\eta,\]
Barbu and Da Prato \cite{BDP-02} provided an ergodicity result. Exponential mixing was proved by Martirosyan \cite{Martirosyan-14} for white-in-time noise using Foias--Prodi estimates, and recently, by Liu, Wei, Xiang, Zhang and Zhao \cite{LWX-24} for localized bounded noise using asymptotic compactness and control theory. The gain of one derivative from the source term is crucial to these studies.

Severe obstacles appear when studying Schr\"{o}dinger and KdV equations, as there is no such regularization from the external force. For Schr\"odinger equations
\[i\partial_tu+\Delta u+iau+f(u)=\eta,\]
Debussche and Odasso \cite{DO-05} obtained polynomial mixing on an interval, under the setting of white-in-time noise, which is based on weak Foias--Prodi estimates. More recently, the authors \cite{CXZZ-25} established exponential mixing on $\T$ for localized noise using nonlinear smoothing and control theory, and Chen, Xiang and Zhang \cite{CXZ-26} tackled the case of highly degenerate Haar noise. See also \cite{BFZ-23,BFZ-24} for ergodic results in higher dimensions. 

For the KdV equation, the unique ergodicity is established very recently by Glatt-Holtz, Martinez and Richards \cite{GHMR-25} for white-in-time noise; in addition, they proved exponential mixing when the damping coefficient is sufficiently large. Whether exponential mixing holds true for
white-in-time forced Schr\"odinger and KdV equations with arbitrary damping remains open.

\subsubsection{Control theory for KdV}

The controllability for KdV equation is investigated by Russell and Zhang \cite{RZ-96}, where the main result is the local controllability on the interval $[0,2\pi]$ with periodic boundary conditions, under a control of the form
\begin{equation}\label{control-form}
g(x) \left(h(t,x)-\int_0^{2\pi} g(y)h(t,y)dy\right).
\end{equation}
Here $g$ is a given smooth cut-off function, and $h$ is the control to be chosen.
This result is then extended to a global version by Laurent, Rosier and Zhang \cite{LRZ-10}, by proving global stabilization for the associated closed-loop system (see also \cite{LLR-15} for the case of Benjamin--Ono equation with the control of form \eqref{control-form}). Recently, Niu, Wang and Xiang \cite{NWX-24} investigated a control localized only on space-time measurable set with positive measure. It is worth mentioning that the control in the current paper has a different structure: it is localized in a space-time domain with non-smooth cut-off $\chi$, and only contains finitely many frequencies. In addition, the control problems for KdV with boundary conditions are very different; see, e.g. \cite{Rosier-97, NX-25}.

\subsection{Organization} 

Section~\ref{Section-preliminary} includes some preliminary results on the well-posedness and nonlinear smoothing for KdV equation, especially Bourgain spaces and normal form transformation. Section~\ref{Section-control} is devoted to the control property, where the main difficulties are observability inequalities in low frequencies, and dealing with the potential terms in high frequencies. In Section~\ref{Section-EM}, we verify three hypotheses in the abstract criterion for mixing, and conclude the proof of the Main Theorem. The Appendix supplies some auxiliary results and proofs.

\subsection{A guide of notations}\label{Section-notation}

We gather here some repeatedly used notations.

\medskip

\noindent\textit{$\bullet$ Functional analysis.} Let $X$ be a Banach space. 
$B_X(x,R)$ denotes the open ball of radius $R$ centered at $x\in X$, and $B_X(R):=B_X(0,R)$. The distance from $x\in X$ to a subset $A\subset X$ is $\dist_X(x,A)$. The space of bounded linear operators from $X$ to $Y$ is $\mathcal{L}(X,Y)$, and $\mathcal{L}(X):=\mathcal{L}(X,X)$. We write $\mathcal{B}(X)$ for the Borel $\sigma$-algebra. The space of bounded continuous functions is $C_b(X)$, equipped with supremum norm $\|\cdot \|_\infty$. And the bounded Lipschitz functions constitute $L_b(X)$, with norm $\|f\|_{L_b(X)}=\|f\|_{\infty}+\Lip(f)$, where $\Lip(f):=\sup_{x\not =y} |f(x)-f(y)|/\|x-y\|$.

\medskip

\noindent\textit{$\bullet$ Function spaces on $\T$.} The Fourier coefficients of $u\colon \T\to \R$ are denoted by $\widehat{u}$. For any $s\in \R$, the Sobolev space $H^s=H^s(\T)$ is equipped with norm $\|u\|_{H^s}=\left(\sum_{k\in\Z}\langle k \rangle^{2s}|\widehat{u}(k)|^2\right)^{1/2}$, where $\langle x\rangle=\sqrt{1+|x|^2}$. 
Let $S_a(t)=e^{-t(\partial_x^3+a)}$ be the $C_0$-group on $H^s$, and $S(t):=S_0(t)$, then
\begin{equation}\label{S_a(t)decay}
\|S_a(t)\|_{\mathcal L(H^s)}=e^{-at}.
\end{equation}

The mean-zero Sobolev space $\dot{H}^s$ consists of $u\in H^s$ with $\widehat{u}(0)=0$. To avoid possible confusion, we clarify that the homogeneous and inhomogeneous Sobolev norms are equivalent for mean-zero functions, and we often use $\|\cdot\|_{\dot{H}^s}$ only to emphasize the mean-zero property. We write $P_m$ for the orthogonal projection from $\dot{H}^s$ to finite-dimensional subspace
\begin{equation}\label{def_Hm}
\dot{H}_m:={\rm span}\{\sin kx, \cos kx ; 1\leq k\le m\}.
\end{equation}

For space-time functions $u\colon \R\times \T\to \R$, the Fourier transform is denoted with $\widehat{u}(\tau,k)$. The Bourgain spaces $X^{s,b},Y^s,Z^s$ and their variants are defined in Section~\ref{Section-standardGWP}.

\medskip

\noindent\textit{$\bullet$ Sobolev spaces on $(s_1,s_2)$.} We introduce $\Delta_{loc}$ as the Laplacian operator on the finite interval $(s_1,s_2)$ with Neumann boundary conditions. The eigenfunctions of $-\Delta_{loc}$ are $1$ and $\beta_k\, (k\in \N^+)$ given by \eqref{orthonormal-basis}, with respect to eigenvalues $0$ and $\lambda_k=\left(\frac{k\pi}{s_2-s_1}\right)^2$. Then for any $s\in (-1/2,1/2)$, the mean-zero Sobolev space $\dot{H}^s (s_1,s_2)$ admits $\{\lambda_k^{-s/2}\beta_k;k\in \N^+\}$ as an orthonormal basis, and can be equipped with the equivalent norm $\|u\|_{\dot{H}^s(s_1,s_2)} = \|(-\Delta_{loc})^{s/2} u \|_{\dot{L}^2(s_1,s_2)}$.

\medskip

\noindent\textit{$\bullet$ Random variables.} Let $X$ be a Polish space (i.e.~separable complete metric space). The law of $X$-valued random variable $\eta$ is $\mathscr{D}(\eta)$, which belongs to the space of Borel probability measures $\mathcal{P}(X)$. The weak convergence in $\mathcal{P}(X)$ is compatible with the dual-Lipschitz distance:
\[\|\mu-\nu\|_L^*:=\sup_{\|f\|_{L_b(X)}\le 1} |\langle f,\mu\rangle-\langle f,\nu\rangle|,\quad \mu,\nu\in \mathcal{P}(X).\]
A coupling between $\mu$ and $\nu$ is a pair of $X$-valued random variables with marginal distributions equal to $\mu$ and $\nu$, respectively. The set of all couplings is denoted by $\mathscr{C}(\mu,\nu)$.

\medskip

\noindent\textit{$\bullet$ Constants.} Various constants $C$ may change from line to line. The dependence on parameters is represented by $C(\cdot)$, which always means a non-decreasing function of such parameter. We write $a\lesssim b$ when there is a universal constant $C>0$ such that $a\leq Cb$. And $a-$ stands for any constant strictly less than $a$. 

\section{Preliminary on Bourgain spaces and nonlinear smoothing}\label{Section-preliminary}

This section reviews some fundamental elements related to the well-posedness of the deterministic KdV equation, including Bourgain spaces, normal form transformation, and nonlinear smoothing. These ingredients also contribute to control property (Section~\ref{Section-HF}) and asymptotic compactness in the probabilistic criterion (Section~\ref{Section-AC}). Owing to these connections to the main objectives of this paper, we provide a fairly detailed exposition.  

For any fixed $T>0$ and $s\in [0,1)$, we consider the following deterministic KdV equation
\begin{equation}\label{Problem-nonlinear}
\begin{cases}
\partial_t u+\partial_x^3 u+a u+\frac{1}{2}\partial_x (u^2)=f,\quad t\in[0,T],\\
u(0,\cdot)=u_0\in \dot{H}^s(\T),
\end{cases}
\end{equation}
where $f\in L^2(0,T;\dot{H}^s(\T))$ denotes an external force with zero spatial mean. Note that the system conserves mass, i.e.~$\int_{\T} u(t,x) dx=0$ for any $t\in [0,T]$. A (mild) solution of \eqref{Problem-nonlinear} is defined as a function $u\in C([0,T];\dot{H}^s)$ satisfying the Duhamel formula
\[u(t)=S_a(t)u_0-\frac{1}{2}\int_0^t S_a(t-r)\partial_x (u^2)(r) dr +\int_0^t S_a(t-r)f(r)dr.\]
The local and global well-posedness of \eqref{Problem-nonlinear} has been extensively studied; see, e.g.~\cite{Bourgain-93,CKSTT-03,KT-06,BIT-11,KV-19}, where various techniques have been employed, such as high-low decomposition, the $I$-method, normal form transformation, and complete integrability. 

In particular, the following version of global well-posedness is useful to the current paper.

\begin{proposition}[Global well-posedness]\label{Prop-uncontrolled}
Given $T>0$ and $s\in [0,1)$, for any $u_0\in \dot{H}^s$ and $f\in L^2(0,T;\dot{H}^{s})$, the initial value problem \eqref{Problem-nonlinear} admits a unique solution $u\in \dot{Y}^s_T\subset C([0,T];\dot{H}^s)$ (see Definition~{\rm\ref{def_Bourgain_YZ}} for the Bourgain space $\dot{Y}^s_T$). 
Moreover, the solution map 
\[\dot{H}^s\times L^2(0,T;\dot{H}^{s})\ni (u_0,f)\mapsto u\in \dot{Y}^s_T\] 
is locally Lipschitz and continuously differentiable.
\end{proposition}

While this result is not new, it is instructive to outline the key arguments:

\begin{itemize}
    \item[\tiny$\bullet$] In Section~\ref{Section-standardGWP}, we recall the definition and basic properties of the Bourgain spaces $X^{s,b}$, $Y^s$ and $Z^s$, which lead to the local well-posedness in $H^s$. The global well-posedness in $L^2$ is then a consequence of energy inequality.

    \item[\tiny$\bullet$] In Section~\ref{sec_normalform_nonlinearsmoothing}, we revisit the normal form transformation, which enables the proof of nonlinear smoothing (Proposition~\ref{prop_nonlinearsmoothing}): for any $u_0\in \dot{L}^2$ and $f\in L^2(0,T;\dot{H}^s)$, we have $u(t)-S_a(t) u_0\in \dot{H}^s$.
    This yields an a priori bound on $\|u\|_{\dot{H}^s}$ provided $u_0\in \dot{H}^s$, from which the global well-posedness in $\dot{H}^s$ follows immediately.
\end{itemize}

\begin{remark}\label{rmk_uu}
Proposition {\rm\ref{Prop-uncontrolled}} only guarantees uniqueness within the particular subspace $\dot{Y}^s_T$ of $C([0,T]; \dot{H}^s)$. It naturally raises the question of ``unconditional uniqueness", namely, whether the solution in $C([0,T];\dot{H}^s)$ is necessarily unique. An affirmative answer can be found in {\rm\cite{BIT-11}}.
\end{remark}

\subsection{Bourgain spaces and local well-posedness}\label{Section-standardGWP}

We follow \cite{Tao-06,ET-16,LRZ-10} to give a quick introduction to Bourgain spaces, which are space-time function spaces taking into account the dispersion relation $\tau=k^3$ of KdV equation.

\begin{definition}[$X^{s,b}$ space]\label{def_Bourgain}
For any $s,b\in \R$, the $X^{s,b}$ norm of tempered distribution $u$ on $\R\times \T$ is defined by
\begin{equation*}
\|u\|_{X^{s,b}}=\|\langle k\rangle^s \langle\tau-k^3\rangle^b \widehat{u}(\tau,k)\|_{l^2_k L^2_\tau}.
\end{equation*}
The Bourgain space $X^{s,b}$ consists of all such $u$ with finite $X^{s,b}$ norm. The subspace $\dot{X}^{s,b}$ of $X^{s,b}$ consists of elements with mean zero in space (i.e.~$\widehat{u}(\tau,0)\equiv 0$), and is equipped with the same norm $\|\cdot\|_{\dot{X}^{s,b}}=\|\cdot\|_{X^{s,b}}$. For any $T>0$, the restricted space to time interval $[0,T]$ is denoted by $X^{s,b}_T$, which is equipped with norm
\[\|u\|_{X^{s,b}_T}=\inf \{\|\tilde{u}\|_{X^{s,b}}; u=\tilde{u}|_{[0,T]\times \T}\}.\]
The mean-zero subspace $\dot{X}^{s,b}_T$ and the restriction $X^{s,b}_I$ to any interval $I\subset \R$ are defined similarly.
\end{definition}

It is readily seen that $X^{s,b}$ is a Hilbert space, and $X^{s_2,b_2}\hookrightarrow X^{s_1,b_1}$ for $s_1\le s_2$ and $b_1\le b_2$. These two properties (with trivial modifications) remain valid for $X^{s,b}_T$. 

A direct computation yields the alternative characterization
\[\|u\|_{X^{s,b}}=\|S(-t)u(t)\|_{H_t^b H_x^s}.\]
where $S(t):=e^{-t\partial_x^3}$. Heuristically, if $b>1/2$ and $u\in X^{s,b}$, then Sobolev embedding gives $S(-t)u(t)\in C_t H_x^s$, and consequently $u\in C_t H_x^s$. This suggests studying the well-posedness of dispersive PDEs in $X^{s,1/2+}$ (see, e.g.~\cite{Bourgain-book} for Schr\"odinger equation). 

Nevertheless, for KdV equation on the torus, the product estimate
\[\|\partial_x (uv)\|_{X^{0,b-1}}\lesssim \|u\|_{X^{0,b}}\|v\|_{X^{0,b}}\]
fails unless $b=1/2$. Meanwhile, $X^{s,1/2}$ is not embedded in $C_t H_x^s$. To circumvent this difficulty, Bourgain \cite{Bourgain-93} introduced modified spaces $Y^s$ and $Z^s$.

\begin{definition}[$Y^s$ and $Z^s$ spaces]\label{def_Bourgain_YZ}
    The Banach space $Y^s$ is defined as a subset of $X^{s,1/2}$ via the norm
\begin{equation*}
\|u\|_{Y^s}:=\|u\|_{X^{s,1/2}}+\|\langle k\rangle^s\widehat u(\tau,k)\|_{l^2_k L^1_\tau}.
\end{equation*}
Here the mixed norm $l^2_k L^1_\tau$ is defined by
\[\|f(\tau,k)\|_{l^2_k L^1_\tau}:=\bigl\| \|f(\tau,k)\|_{L^1_\tau}\bigr\|_{l^2_k}.\]
And the Banach space $Z^s$ is defined as a subset of $X^{s,-1/2}$ via the norm
\begin{equation}\label{def_Z^s}
\|u\|_{Z^s}:=\|u\|_{X^{s,-1/2}}+\|\langle k\rangle^s\langle \tau-k^3\rangle^{-1}\widehat u(\tau,k)\|_{l^2_k L^1_\tau}.
\end{equation}
The restricted spaces $Y^s_T$ and $Z^s_T$, together with their mean-zero counterparts $\dot{Y}^s_T$ and $\dot{Z}^s_T$, are defined analogously to $X^{s,b}_T$ and $\dot{X}^{s,b}_T$.
\end{definition}

The following basic properties make Bourgain spaces useful in the study of KdV equation.

\begin{lemma}\label{lem_Bougainproperty}
    Given $T>0$, $s>-1/2$ and $b\in (1/2,1)$, the following properties are valid.
	\begin{enumerate}
        \item[$(1)$] The identity maps induce continuous embeddings 
        \[X^{s,b}_T\hookrightarrow C([0,T];H^s),\quad Y^s_T\hookrightarrow C([0,T];H^s),\quad L^2(0,T;H^s)\hookrightarrow Z^s_T.\]
        
		\item[$(2)$] There exists a constant $C>0$, such that for any $u_0\in H^s$, 
		\[\|S_a(t)u_0\|_{Y^s_T}\leq C\|u_0\|_{H^s}.\]
        In addition, for any $f\in Z_T^s$ and $g\in X^{s,b}_T$,
        \[
		\left\|\int_0^t S_a(t-r)f(r)dr\right\|_{Y^s_T}\leq C\|f\|_{Z^s_T},\quad \left\|\int_0^t S_a(t-r)g(r)dr\right\|_{X^{s,b}_T}\leq C\|g\|_{X^{s,b-1}_T}.\]

        \item[$(3)$] There exists a constant $C>0$ (independent of $T$), such that for any $u,v\in \dot{X}_T^{s,1/2}$,
        \[\|\partial_x(uv)\|_{\dot{Z}^s_T}\le C (T^{1/6-}\wedge 1) \|u\|_{\dot{X}^{s,1/2}_T}\|v\|_{\dot{X}^{s,1/2}_T}\]
        (here we denote $b\pm=b\pm\varepsilon$ with $0<\varepsilon\ll 1$, and $a\wedge b:=\min \{a,b\}$).

        \item[$(4)$] Assume the intervals $I_1,\dots ,I_n$ are relatively open in $[0,T]$, and constitute a finite covering of $[0,T]$. Then there exists a constant $C>0$, such that if $u_j\in Y^s_{\overline{I_j}}$ and $u|_{I_j}=u_j$ for $1\le j\le n$, then
        \[\|u\|_{Y^s_T}\le C\sum_{j=1}^n \|u_j\|_{Y^s_{\overline{I_j}}}.\]
	\end{enumerate}
\end{lemma}

\begin{proof}[\bf Proof]
    The first two embeddings in (1) are easy consequences of Fourier inversion formula; see, e.g.~\cite[Lemma~3.9 and Section~3.2.2]{ET-16} for details. To verify the third embedding in (1), it suffices to prove $L^2(\R;H^s)\hookrightarrow Z^s$ and then take restrictions to time $[0,T]$. Indeed, owing to $X^{s,0}=L^2(\R;H^s)$, Cauchy--Schwarz inequality and $\|\langle\tau\rangle^{-1}\|_{L^2_\tau}<\infty$, we have
    \begin{align*}
        \|f\|_{Z^s}&\le \|f\|_{L^2(\R;H^s)}+\bigl\|\langle k\rangle^s \|\langle \tau-k^3\rangle^{-1}\|_{L^2_\tau} \|\widehat{f}(\tau,k)\|_{L^2_\tau}\bigr\|_{l^2_k}\\
        &\le \|f\|_{L^2(\R;H^s)}+C\|\langle k\rangle^s \widehat{f}(\tau,k)\|_{l_k^2 L_\tau^2}=C\|f\|_{L^2(\R;H^s)}.
    \end{align*}
    Next, (2) is standard when $a=0$ (i.e.~without damping); see, e.g.~\cite[Lemma~3.12, Lemma~3.15 and Lemma~3.16]{ET-16}. Meanwhile, the proof of (2) when $a>0$ can be adapted from \cite[Lemma~4.4]{LRZ-10}\footnote{In fact, \cite{LRZ-10} only considered $f=\partial_x (uv)$ and $s\ge 0$. However, the proof for general $f$ and $s>-1/2$ is verbatim.}. Furthermore, (3) is a special case of \cite[Theorem 3.19 and Corollary 3.20]{ET-16}.

    Finally, we turn to (4), which is believed to be standard (but we cannot find an explicit reference). The key point is the time-localization estimate: for any $\varphi\in C_c^\infty(\R)$, there exists a constant $C>0$, such that
    \[\|\varphi(t) v\|_{Y^s}\le C \|v\|_{Y^s},\quad \forall v\in Y^s.\]
    Indeed, this is well-known for $X^{s,b}$ norm; see, e.g.~\cite[Lemma 2.11]{Tao-06}. The proof therein can be trivially adapted to $Y^s$. With this inequality in hand, the rest of the proof is easy by partition of unity. We omit further details.
\end{proof}

These properties enable one to derive the local well-posedness of \eqref{Problem-nonlinear} in $\dot{Y}^s_\delta\subset C([0,\delta];\dot{H}^s)$ from some $\delta>0$, based on the standard fixed-point argument. In addition, multiplying \eqref{Problem-nonlinear} by $u$ and integrating over $\T$, we obtain the energy inequality
\begin{equation*}
\frac{d}{dt}\|u\|^2_{\dot{L}^2}+a\|u\|^2_{\dot{L}^2}\leq\frac{1}{a}\|f\|^2_{\dot{L}^2}.
\end{equation*}
Together with the Gronwall inequality, we find that for any $t\in [0,\delta]$,
\begin{equation}\label{a priori-estimate-1}
\|u(t)\|^2_{\dot{L}^2}\le e^{-at} \|u_0\|_{\dot{L}^2}^2+C\|f\|_{L^2(0,T; \dot{L}^2)}^2.
\end{equation}
This a priori estimate implies the global well-posedness at the scale of $L^2$ by continuity argument. Moreover, we have a qualitative bound in a more refined space $\dot{Y}^0_T$, namely
\begin{equation}\label{a priori-estimate-2}
\|u\|_{\dot{Y}^0_T}\le C(\|u_0\|_{\dot{L}^2},\|f\|_{L^2(0,T;\dot{L}^2)}).
\end{equation}
Indeed, the local well-posedness implies this estimate on any $[t,t+\delta]$ with $\delta>0$ depending only on $\|u_0\|_{\dot{L}^2}$ and $\|f\|_{L^2(0,T;\dot{L}^2)}$, and then Lemma~\ref{lem_Bougainproperty}(4) leads to the global version.

\medskip

Before ending this subsection, we collect some well-known properties of $X^{s,b}$ for easy reference (see, e.g.~\cite[Section 3.1]{LRZ-10}), which will be used in the remainder of this paper.

\begin{lemma}\label{Lemma-Bourgain}
The following properties are valid.
\begin{enumerate}
\item[$(1)$] If $T>0$, $r_1<r_2$ and $b_1<b_2$, then $X^{r_2,b_2}_T$ is compactly embedded into $X^{r_1,b_1}_T$.

\item[$(2)$] If $b\in [-1,1]$, then for any $\varphi\in C^\infty(\T)$, the multiplication 
\[X^{r,b}\ni u\mapsto \varphi(x)u\in X^{r-2|b|,b}\]
is bounded. Corresponding property remains valid for $X^{r,b}_T$.

\item[$(3)$] There exists a constant $C>0$ such that for any $u\in X^{0,1/3}$,
\[\|u\|_{L^4_{t,x}}\leq C\|u\|_{X^{0,1/3}}.\]

\item[$(4)$] If $-1/2<b_1<b_2<1/2$, then there exists a constant $C>0$, such that for any $T\in (0,1)$ and $u\in X^{r,b_2}_T$,
\begin{equation*}
\|u\|_{X^{r,b_1}_T}\le C T^{b_2-b_1} \|u\|_{X^{r,b_2}_T}.
\end{equation*}
\end{enumerate}
\end{lemma}

\subsection{Normal form transformation and nonlinear smoothing}\label{sec_normalform_nonlinearsmoothing}

Now we turn to normal form transformation, which uses integration by parts in time to produce large denominators from dispersion, at the cost of increasing the order of nonlinearity. As a result, the nonlinear part of solution can be represented by several multilinear operators with higher regularity. This method was applied to the well-posedness in \cite{BIT-11} and nonlinear smoothing in \cite{ET-13-kdv}.

Following \cite[Section~4.1.2]{ET-16}, let us define the mean-zero bilinear/trilinear operators $\mathcal{B},\mathcal I,\mathcal D$ for functions defined on $\T$ via their Fourier modes: for any $k\in \Z\setminus\{0\}$, set
\begin{equation}\label{Multilinear-operator}
\begin{aligned}
&\widehat{\mathcal{B}(f,g)}(k)=-\frac{1}{6}\sum_{k_1+k_2=k}\frac{\widehat{f}(k_1)\widehat{g}(k_2)}{k_1k_2},\\ 
&\widehat{\mathcal I(f,g,h)}(k)=\frac{i}{6}\left(
-\frac{1}{k}\widehat{f}(k)\widehat{g}(-k)\widehat{h}(k)+\widehat{f}(k)\sum_{|j|\neq |k|}\frac{1}{j}\widehat{g}(j)\widehat{h}(-j)
\right),\\
&\widehat{\mathcal D(f,g,h)}(k)=\frac{i}{6}\sum_{\substack{k_1+k_2+k_3=k\\(k_1+k_2)(k_2+k_3)(k_1+k_3)\not =0}}\frac{\widehat{f}(k_1)\widehat{g}(k_2)\widehat{h}(k_3)}{k_1}.
\end{aligned}
\end{equation}
The advantage is that the large denominators indicate smoothing effects for these operators.

The following formula, known as normal form transformation, provides an alternative expression for the solution. We also refer the reader to \eqref{identity-4} later for a similar formula concerning the linearized KdV equation, whose proof is in the same spirit.

\begin{lemma}[{\cite[Proposition 4.4 and (4.15)]{ET-16}\footnote{In fact, \cite{ET-16} also considered KdV equation with potential term $(Vu)_x$. The formula \eqref{Identity-DBP} below coincides with
(4.15) therein, while the notation of the term $\tilde{\rho}$ is replaced by $\mathcal{I}(u,u,u)$ (one easily see that these two expressions are equal, since $\widehat{u}(-j)=\overline{\widehat{u}(j)}$ for real $u$ and the summation over $j$ is canceled out). This more general notation $\mathcal{I}(f,g,h)$ is used for future convenience when studying control property in Section~\ref{Section-control}.}}]\label{Lemma-DBP}
Assume that $u\in C([0,T];\dot{L}^2)$ is a solution of \eqref{Problem-nonlinear} with $u_0\in \dot{L}^2$ and $f\in L^2(0,T;\dot{L}^2)$. Then for any $t\in [0,T]$,
\begin{equation}\label{Identity-DBP}
\begin{aligned}
&u(t)-S_a(t)u_0=S_a(t)\mathcal{B}(u_0,u_0)-\mathcal{B}(u(t),u(t))\\
&\quad +\int_0^t S_a(t-r) \left[-a\mathcal{B}(u,u)+2\mathcal{B}(u,f)+\mathcal I(u,u,u)+\mathcal{D}(u,u,u)+f(r)\right]dr.
\end{aligned}
\end{equation}
\end{lemma}

Then we state the smoothing properties for $\mathcal{B},\mathcal{I},\mathcal{D}$. Although the operators are defined in a slightly more general way than \cite{ET-16}, the proof is still valid and hence omitted.

\begin{lemma}[{\cite[Lemma 4.5 and Proposition 4.7]{ET-16}}]\label{Lemma-BID}
Given $s\ge 0$ and $\sigma\in (0,1)$, the following properties are valid.
\begin{enumerate}
\item[$(1)$] There exists a constant $C>0$ such that for any $f,g,h\colon \T\to \R$ of mean zero,
\begin{equation*}
\begin{aligned}
& \|\mathcal{B}(f,g)\|_{\dot{H}^{s+\sigma}}\leq C\|f\|_{\dot{H}^s}\|g\|_{\dot{H}^s},\\
&\|\mathcal I(f,g,h)\|_{\dot{H}^{s+\sigma}}\leq C\|f\|_{\dot{H}^{s+\sigma}}\|g\|_{\dot{H}^{s}}\|h\|_{\dot{H}^s}.
\end{aligned}
\end{equation*}
In particular, if $f=g=h$, then the term $\mathcal{I}(f,f,f)$ reduces to
\[\widehat{\mathcal{I}(f,f,f)}(k)=-\frac{i}{6}\cdot \frac{1}{k} \widehat{f}(k)|\widehat{f}(k)|^2,\quad k\in \Z\setminus \{0\},\]
and satisfies a stronger estimate
\[\|\mathcal I(f,f,f)\|_{\dot{H}^{s+\sigma}}\leq C\|f\|_{\dot{H}^s}^3.\]
		
\item[$(2)$] For any $\varepsilon>0$ sufficiently small, there exists a constant $C>0$, such that for any $f,g,h\colon [0,T]\times \T\to \R$ of mean zero in space,
\[\|\mathcal D(f,g,h)\|_{\dot{X}_T^{s+\sigma,-1/2+\varepsilon}}\leq C\|f\|_{\dot{X}^{s,1/2}_T}\|g\|_{\dot{X}^{s,1/2}_T}\|h\|_{\dot{X}^{s,1/2}_T}.\]
\end{enumerate}
\end{lemma}

With the above results, we obtain the following nonlinear smoothing property. The proof is similar to \cite[Theorem 4.3]{ET-16}. For the reader's convenience, we carry out the details.

\begin{proposition}[Nonlinear smoothing]\label{prop_nonlinearsmoothing}
Given $T>0$ and $s\in (0,1)$, and let $u\in Y^0_T$ be the solution of \eqref{Problem-nonlinear} with initial data $u_0\in \dot{L}^2$ and source term $f\in L^2(0,T;\dot{H}^s)$, then
\[\|u(t)-S_a(t) u_0\|_{\dot{H}^s}\le C(\|u_0\|_{\dot{L}^2},\|f\|_{L^2(0,T;\dot{H}^s)}),\quad \forall t\in [0,T].\]
\end{proposition}

\begin{proof}[\bf Proof]
    We take the parameters $s$ and $\sigma$ in Lemma~\ref{Lemma-BID} as $0$ and $s$, respectively. In particular, Lemma~\ref{Lemma-BID}(2) holds for some constant $\varepsilon>0$. Note that $\dot{Y}^0_T$ is embedded into $L^\infty(0,T;\dot{L}^2)$ and $\dot{X}_T^{0,1/2}$. Taking also Lemma~\ref{lem_Bougainproperty}(1)(2) and \eqref{S_a(t)decay}, \eqref{a priori-estimate-2}, \eqref{Identity-DBP} into account, it follows that
    \begin{align*}
        \|u(t)-S_a(t)u_0\|_{\dot{H}^s}&\lesssim e^{-at}\|\mathcal{B}(u_0,u_0)\|_{\dot{H}^s}+\|\mathcal{B}(u(t),u(t))\|_{\dot{H}^s}\\
        &\quad +\int_0^t e^{-a(t-r)} (\|\mathcal{B}(u,u)\|_{\dot{H}^s}+\|\mathcal{B}(u,f)\|_{\dot{H}^s}+\|\mathcal{I}(u,u,u)\|_{\dot{H}^s})\, dr\\
        &\quad +\left\|\int_0^t S_a(t-r) \mathcal{D}(u,u,u) dr\right\|_{\dot{X}^{s,1/2+\varepsilon}_T}+\left\|\int_0^t S_a(t-r) f(r) dr\right\|_{\dot{Y}^s_T}\\
        &\lesssim \|u_0\|_{\dot{L}^2}^2+\|u(t)\|_{\dot{L}^2}^2+\int_0^t e^{-a(t-r)}(\|u\|_{\dot{L}^2}^2+\|u\|_{\dot{L}^2}\|f\|_{\dot{L}^2}+\|u\|_{\dot{L}^2}^3) dr\\
        &\quad +\|\mathcal{D}(u,u,u)\|_{\dot{X}^{s,-1/2+\varepsilon}_T}+\|f\|_{\dot{Z}^s_T}\\
        &\le C(\|u_0\|_{\dot{L}^2},\|f\|_{L^2(0,T;\dot{H}^s)})+\|u\|_{X^{0,1/2}_T}^3\le C(\|u_0\|_{\dot{L}^2},\|f\|_{L^2(0,T;\dot{H}^s)}).
    \end{align*}
    Now the proof is complete.
\end{proof}

Specifically, if $u_0\in \dot{H}^s$, then nonlinear smoothing implies an a priori bound
\begin{equation}\label{a priori-estimate-3}
    \|u(t)\|_{\dot{H}^s}\le \|u(t)-S_a(t)u_0\|_{\dot{H}^s}+\|S_a(t)u_0\|_{\dot{H}^s}\le C(\|u_0\|_{\dot{H}^s},\|f\|_{L^2(0,T;\dot{H}^s)}).
\end{equation}
Thus we can extend the local $H^s$ solution to a global one by continuity argument. To fulfill the justification of Proposition~\ref{Prop-uncontrolled}, we point out that the local Lipschitz continuity of solution map on small intervals $[t,t+\delta]$ is a by-product of the fixed point argument, which then extends to $[0,T]$ by Lemma~\ref{lem_Bougainproperty}(4). In addition, the $C^1$ regularity of solution map follows from the well-posedness of linearized equation; see Proposition~\ref{Prop-linear-1}.

\section{Stabilization along trajectory by localized force}\label{Section-control}

This section is devoted to a control property associated with the coupling condition in the mixing criterion. Let us introduce the controlled KdV equation
\begin{equation}\label{Control-nonlinearproblem}
\begin{cases}
\partial_t u+\partial^3_x u+au+\tfrac{1}{2}\partial_x(u^2)=h+\chi\mathcal{P}_N\xi,\quad t\in[0,1],\\
u(0,\cdot)=u_0.
\end{cases}
\end{equation}
Here, $h$ is a given external force and $\chi \mathcal{P}_N \xi$ is the control term, in which $\chi$ is the space-time localization \eqref{eq-chi}, $\mathcal{P}_N$ is the orthogonal projection to the low frequencies \eqref{eq-PN}, and $\xi\in L^2(t_1,t_2;\dot{L}^2(s_1,s_2))$ is the control to be chosen. Note that for any $\varepsilon>0$,
\[\chi \mathcal{P}_N \xi \in L^2(0,1; \dot{H}^{1/2-\varepsilon}(\T)).\]
When $\xi\equiv 0$, the system reduces to the uncontrolled one
\begin{equation}\label{Uncontrolled-Problem}
\begin{cases}
\partial_t w+\partial^3_xw+aw+\tfrac{1}{2}\partial_x(w^2)=h,\quad t\in[0,1],\\
w(0,\cdot)=w_0.
\end{cases}
\end{equation}
In addition to its importance to the study of mixing, the following result is also of independent interest in control theory (referred to as local stabilization).

\begin{theorem}[Local stabilization]\label{Theorem-control}

Given $q\in (e^{-a},1)$, $\sigma\in (0,1/4)$ and $R>0$, there exist constants $d\in (0,1)$ and $N\in \N^+$, such that the following assertions hold.

\begin{enumerate}
\item[$(1)$] If $w\in \dot{Y}^{1/4+\sigma}_1$ is the solution of uncontrolled system \eqref{Uncontrolled-Problem} with initial data $w_0$ and source term $h$ satisfying
\[\|w_0\|_{\dot{H}^{1/4+\sigma}}< R,\quad \|h\|_{L^2(0,1;\dot{H}^{1/4+\sigma})}<R,\]
then for any $u_0\in \dot{L}^2(\T)$ with $\|u_0-w_0\|_{\dot{L}^2}\le d$, there exists a control $\xi$ such that 
\begin{equation}\label{control-bound}
\int_{t_1}^{t_2} \|\xi(t)\|_{\dot{L}^2(s_1,s_2)}^2 dt\le C(q,\sigma,R) \|u_0-w_0\|_{\dot{L}^2}^2,
\end{equation}
and the solution $u\in \dot{Y}^0_1$ of controlled system \eqref{Control-nonlinearproblem} satisfies
\begin{equation}\label{Squeezing-0}
\|u(1)-w(1)\|_{\dot{L}^2}\leq q \|u_0-w_0\|_{\dot{L}^2}.
\end{equation}

\item[$(2)$] The control $\xi$ in {\rm(1)} can be represented as
\[\xi=\Phi(w_0,h)(u_0-w_0),\]
where the map (determined by $q,\sigma,R$)
\[\Phi\colon B_{\dot{H}^{1/4+\sigma}}(R)\times B_{L^2(0,1;\dot{H}^{1/4+\sigma})}(R)\rightarrow \mathcal{L}(\dot{L}^2(\T);L^2(t_1,t_2;\dot{L}^2(s_1,s_2)))\]
is Lipschitz and continuously differentiable.
\end{enumerate}
\end{theorem}

We emphasize the ``extra regularity'' of the reference trajectory $w\in \dot{Y}^{1/4+\sigma}_1$ is crucial to the proof, while both the closeness between initial data $\|u_0-w_0\|_{\dot{L}^2}\le d$ and the stabilization property \eqref{Squeezing-0} are characterized in $\dot{L}^2$.

Theorem~\ref{Theorem-control} follows from the corresponding stabilization result for linearized system via a perturbative argument (cf.~\cite{Shi-15,LWX-24,CXZZ-25}); see Section~\ref{Section-proofcontrol}. Indeed, the linearized system along $w$ is
\begin{equation}\label{Control-problem}
\begin{cases}
\partial_t v+\partial^3_xv+av+\partial_x(wv)=\chi\mathcal{P}_N\xi,\quad t\in[0,1],\\
v(0,\cdot)=v_0.
\end{cases}
\end{equation}
In view of Proposition~\ref{Prop-linear-1}, for $w\in \dot{Y}^{1/4+\sigma}_1$ and $v_0\in \dot{H}^s\, (s\in [0,1/4+\sigma))$, this linearized system admits a unique solution $v\in \dot{Y}^s_1$. This solution $v$ is denoted by 
\[v=\mathcal V_{w}(v_0,\chi\mathcal{P}_N\xi).\] 
The proof consists of three parts:
\begin{itemize}
\item[\tiny$\bullet$] In Section~\ref{Section-LF}, we study the low-frequency controllability based on observability and HUM. The method is analogous to \cite{RZ-96,RZ-06,LRZ-10}, while the techniques differ since we are dealing with linearized KdV in $L^2$. Specifically, the observability inequality (Lemma~\ref{Lemma-fullobs}) and unique continuation property (Proposition~\ref{Prop-UCP}) are novel.

\item[\tiny$\bullet$] In Section~\ref{Section-HF}, we adapt the idea from \cite{CXZZ-25} to derive high-frequency dissipation via nonlinear smoothing (Proposition~\ref{Prop-linearized_nonlinearsmooothing}). To this end, we extend normal form transformation to linearized KdV, and exploit a bilinear estimate (Lemma~\ref{Lemma-smoothing-1}) to gain regularity.

\item[\tiny$\bullet$] In  Section~\ref{Section-proofcontrol}, we combine information on low and high frequencies to conclude a 
stabilization property for linearized system \eqref{Control-problem}, and then establish Theorem~\ref{Theorem-control} with a perturbative argument.
\end{itemize} 

\settocdepth{subsection}

\subsection{Low-frequency controllability via observability inequality}\label{Section-LF}

This subsection is concerned with the low-frequency controllability for \eqref{Control-problem}, in a slightly more general space $\dot{H}^s$ with $s\in [0,1/4+\sigma)$. In what follows, $P_m\colon \dot{H}^s(\T)\to \dot{H}_m$ is the orthogonal projection to low Fourier modes, where $\dot{H}_m$ is defined by \eqref{def_Hm} (we inform the reader to distinguish $\mathcal{P}_N$ and $P_m$).

\begin{proposition}[Low-frequency controllability]\label{Prop-LFcontrol}
Given $\sigma\in (0,1/4)$, $s\in [0,1/4+\sigma)$, $R>0$ and $m\in \N^+$, there exists a constant $N\in \N^+$, such that the following assertions hold.
\begin{enumerate}
\item[$(1)$] If $\|w\|_{\dot{Y}^{1/4+\sigma}_1}<R$, then for any $v_0\in \dot{H}^s(\T)$, there exists a control $\xi$ such that
\begin{equation*}
\int_{t_1}^{t_2} \|\xi(t,\cdot)\|_{\dot{H}^s(s_1,s_2)}^2 dt\le C(\sigma,s,R)\|v_0\|_{\dot{H}^s}^2,
\end{equation*}
and the solution $v\in \dot{Y}^s_1$ of linearized system \eqref{Control-problem} satisfies
\begin{equation*}
P_m v(1)=0.
\end{equation*}

\item[$(2)$] The control $\xi$ in {\rm (1)} can be represented as
\[\xi=\Lambda(w)(v_0),\]
where the map (determined by $\sigma,s,R,m$)
\[\Lambda\colon B_{\dot{Y}^{1/4+\sigma}_1}(R)\to \mathcal{L}(\dot{H}^s;L^2(t_1,t_2;\dot{H}^s(s_1,s_2)))\]
is Lipschitz and continuously differentiable.
\end{enumerate}
\end{proposition}

We point out $w$ need not be a solution of \eqref{Uncontrolled-Problem} (only its $\dot{Y}^{1/4+\sigma}_1$ norm is important to the proof), and the constant $C(\sigma,s,R)$ in \eqref{control-bound} does not depend on $m$.

It is well known in control theory that the controllability is equivalent to observability. In an abstract sense, it means that a bounded linear operator $T\colon X\to Y$ is surjective if and only if $\|T^* y^*\|\ge C\|y^*\|$ for any $y^*\in Y^*$. This suggests investigating the dual of the map
\[L^2(s_1,s_2;\dot{H^s}(s_1,s_2))\ni \xi\mapsto P_m v(1)\in \dot{H}_m.\]
To this end, we introduce the backward system (here $\widehat{f}$ stands for spatial Fourier modes of $f$)
\begin{equation}\label{Adjoint-problem}
\begin{cases}
\partial_t \varphi+\partial^3_x\varphi-a\varphi+[w\partial_x\varphi-\widehat{w\partial_x\varphi}(0)]=0,\quad t\in[0,1],\\
\varphi(1,\cdot)=\varphi_1.
\end{cases}
\end{equation}
In view of Proposition~\ref{Prop-linear-2}, it admits a unique solution $\varphi\in \dot{Y}^{-s}_1$ for every $\varphi_1\in \dot{H}^{-s}$, as long as $s\in [0,1/4+\sigma)$ and $w\in \dot{Y}^{1/4+\sigma}_1$. 
The solution is denoted by 
\[\varphi=\mathcal U_w(\varphi_1).\]

System \eqref{Adjoint-problem} is regarded as the adjoint of \eqref{Control-problem} for the following reason. If $v=\mathcal V_w(v_0,\chi \mathcal{P}_N \xi)$ with $\xi\in L^2(t_1,t_2;\dot{L}^2(s_1,s_2))$, and $\varphi=\mathcal U_w(\varphi_1)$ with $\varphi_1\in \dot{L}^2$, then
\begin{equation}\label{Dual-identity}
\begin{aligned}
(v(1),\varphi_1)_{\dot{L}^2}-(v_0,\varphi(0))_{\dot{L}^2}&=\int_0^1\frac{d}{dt}(v(t),\varphi(t))_{\dot{L}^2} dt\\
& =\int_0^1(-\partial^3_xv-av-\partial_x(wv)+\chi\mathcal{P}_N\xi,\varphi)_{\dot{L}^2} dt\\
&\quad +\int_0^1\left(v,-\partial^3_x\varphi+a\varphi-\left[w\partial_x\varphi-\widehat{w\partial_x\varphi}(0)\right]\right)_{\dot{L}^2}dt\\
& =\int_0^1 (\mathcal{P}_N\xi, \chi \varphi)_{\dot{L}^2} dt=\int_0^1 (\xi,\mathcal{P}_N G\varphi)_{\dot{L}^2(s_1,s_2)} dt,
\end{aligned}
\end{equation}
where the operator $G$ is given by
\begin{equation}\label{G-definition}
    G\colon L^2(0,1;\dot{H}^{-s})\to L^2(t_1,t_2;\dot{H}^{-s}(s_1,s_2)),\quad G\psi:=\chi\left(\psi-\tfrac{2\pi}{|s_1-s_2|}\widehat{\mathbf{1}_{(s_1,s_2)}\psi}(0)\right).
\end{equation}

By utilizing the dual identity \eqref{Dual-identity} and HUM (see Appendix~\ref{Section-HUM}), it is not hard to reduce Proposition~\ref{Prop-LFcontrol} to the following observability inequality.

\begin{lemma}[Truncated observability]\label{Lemma-truncatedobs}
Given $\sigma\in(0,1/4)$, $s\in[0,1/4+\sigma)$ and $R>0$, there exists a constant $C_0>0$ such that for every $m\in\N^+$, one can find $N\in\N^+$ such that if $\|w\|_{\dot{Y}^{1/4+\sigma}_1}<R$, $\varphi_1\in \dot{H}_m$ and $\varphi=\mathcal{U}_w(\varphi_1)$, then
\begin{equation}\label{Truncated-obs}
\int_{t_1}^{t_2}\|\mathcal{P}_NG\varphi(t)\|_{\dot{H}^{-s}(s_1,s_2)}^2dt\geq C_0\|\varphi_1\|^2_{\dot{H}^{-s}}.
\end{equation}
\end{lemma}

This truncated inequality follows from its ``full version'' (i.e.~$m=N=\infty$).

\begin{lemma}[Full observability]\label{Lemma-fullobs}
Given $R>0$, $\sigma\in (0,1/4)$ and $s\in[0,1/4+\sigma)$, there exists a constant $C_1>0$ such that if $\|w\|_{\dot{Y}^{1/4+\sigma}_1}<R$, $\varphi_1\in \dot{H}^{-s}$ and $\varphi=\mathcal{U}_w(\varphi_1)$, then
\begin{equation}\label{Full-obs}
\int_{t_1}^{t_2}\|G\varphi(t)\|_{\dot{H}^{-s}(s_1,s_2)}^2dt\geq C_1\|\varphi_1\|^2_{\dot{H}^{-s}}.
\end{equation}
\end{lemma}

Taking the full observability for granted, we can justify the truncated counterpart.

\begin{proof}[{\bf Proof of Lemma~\ref{Lemma-truncatedobs}}]
Define an auxiliary operator $\widetilde{G}$ by
\[\widetilde{G}\psi:=\psi-\tfrac{2\pi}{|s_1-s_2|}\widehat{\mathbf{1}_{(s_1,s_2)}\psi}(0).\]
In other words, $\widetilde{G}$ differs from $G$ by removing the localization $\chi$. 
Owing to $G\varphi=\chi \widetilde{G}\varphi$, we have
\begin{equation}\label{truncated-bound-1}
\begin{aligned}
&\int_{t_1}^{t_2}\|G\varphi(t)\|_{\dot{H}^{-s}(s_1,s_2)}^2dt\\
=&\int_{t_1}^{t_2}\|\mathcal{P}_NG\varphi(t)\|_{\dot{H}^{-s}(s_1,s_2)}^2dt+\int_{t_1}^{t_2}\|(Id-\mathcal{P}_N)\chi \widetilde{G}\varphi(t)\|_{\dot{H}^{-s}(s_1,s_2)}^2dt.
\end{aligned}
\end{equation}

To deal with the second integral, we have the following three observations. Let us fix two small parameters $r>0$ and $b\in (0,1/2)$ so that $r+2b<1/4$. 
\begin{itemize}
    \item[\tiny$\bullet$] The embedding $X_1^{r,b}\hookrightarrow X_1^{0,0}$ is compact (Lemma~\ref{Lemma-Bourgain}(1)), and hence the composition
    \[X^{r,b}_1\hookrightarrow X_1^{0,0}= L^2(0,1;L^2(\T))\overset{\chi}\to L^2(t_1,t_2;H^{-s}(s_1,s_2))\]
    is also compact. As a result, there exists a sequence $\delta_N\rightarrow 0^+$ such that
    \[\|(Id-\mathcal{P}_N)\chi g\|_{L^2(t_1,t_2;\dot{H}^{-s}(s_1,s_2))}^2\leq \delta_N\|g\|_{X^{r,b}_1}^2,\]
    for any $g\in X^{r,b}_1$ with mean zero on $(s_1,s_2)$.

    \item[\tiny$\bullet$] Clearly, $\widetilde{G}$ is a bounded linear operator from $X^{r,b}_1$ into itself.

    \item[\tiny$\bullet$] Since $\dot{H}_m$ is finite dimensional and consists of smooth functions, all Sobolev norms on $\dot{H}^m$ are equivalent. In particular, 
    \[\|\varphi_1\|_{\dot{H}^{1/4}}\le C(m)\|\varphi_1\|_{\dot{H}^{-s}},\quad \forall \varphi_1\in \dot{H}_m.\]
    Moreover, due to $b<1/2$ and $r+2b<1/4$, the solution $\varphi=\mathcal{U}_w(\varphi_1)\in \dot{Y}^{1/4}_1$ with $\varphi_1\in \dot{H}_m$ satisfies the a priori estimate
    \[\|\varphi\|_{X_1^{r+2b,b}}\le C\|\varphi\|_{X^{1/4,1/2}_1}\le C\|\varphi\|_{Y^{1/4}}\le C \|\varphi_1\|_{\dot{H}^{1/4}}.\]
    Here the constant $C$ depends on $\|w\|_{\dot{Y}^{1/4+\sigma}_1}$, which is in turn determined by $R$.
\end{itemize}
Gathering these observations, and applying Lemma~\ref{Lemma-fullobs}, we have
\begin{align*}
\int_{t_1}^{t_2}\|(Id-\mathcal{P}_N)\chi \widetilde{G}\varphi(t)\|_{\dot{H}^{-s}(s_1,s_2)}^2dt&\leq \delta_N\|\widetilde{G}\varphi\|_{X^{r,b}_1}^2\leq C\delta_N\|\varphi\|_{X^{r+2b,b}_1}^2\\
&\leq C\delta_N\|\varphi_1\|_{\dot{H}^{1/4}}^2 \leq C(m)\delta_N\|\varphi_1\|_{\dot{H}^{-s}}^2\\
&\leq C(m)\delta_N\int_0^T\|G\varphi(t)\|_{\dot{H}^{-s}(s_1,s_2)}^2dt,
\end{align*}
Let us take $N=N(m)$ sufficiently large so that $C(m)\delta_N\le 1/2$.
Substituting into \eqref{truncated-bound-1} yields
\[\int_{t_1}^{t_2}\|G\varphi(t)\|_{\dot{H}^{-s}(s_1,s_2)}^2dt\leq C \int_{t_1}^{t_2}\|\mathcal{P}_NG\varphi(t)\|_{\dot{H}^{-s}(s_1,s_2)}^2dt+\frac{1}{2}\int_{t_1}^{t_2} \|G\varphi(t)\|_{\dot{H}^{-s}(s_1,s_2)}^2dt.\]
The second term on the right-hand side can be absorbed in the left-hand side. Thus we conclude \eqref{Truncated-obs} by using again \eqref{Full-obs}, where the constant $C_0$ is taken as $C_0=C_1/2C$.
\end{proof}

The derivation from truncated observability to the low-frequency controllability, i.e.~Proposition~\ref{Prop-LFcontrol}, follows from the standard HUM argument (cf.~\cite{LWX-24,CXZZ-25}). For the sake of completeness, we provide a sketched proof in Appendix~\ref{Section-HUM}.
\vspace{2mm}

It remains to demonstrate Lemma~\ref{Lemma-fullobs}, constituting the most complicated part of this subsection. The proof follows from a similar method as in \cite{RZ-96,RZ-06,LRZ-10}, where the equations are essentially the nonlinear KdV, and the control is localized by a smooth cut-off. Heuristically, in view of the compactness-uniqueness method, if this lemma fails, then we can construct a sequence such that $\|\varphi_1^n\|_{\dot{H}^{-s}}=1$ and $\|G\varphi^n\|_{L^2(t_1,t_2;\dot{H}^{-s}(s_1,s_2))}\to 0$. Extracting a convergent subsequence, we find a limit solution $\varphi$ with $\|\varphi_1\|_{\dot{H}^{-s}}=1$ and $G\varphi=0$, which implies (after suitable modifications) that $\varphi(t,x)=0$ for any $t\in (t_1,t_2)$ and $x\in(s_1,s_2)$. And finally a unique continuation property implies $\varphi\equiv 0$, which contradicts $\|\varphi_1\|_{\dot{H}^n}=1$. Nevertheless, here some technical delicacies arise, primarily due to the low regularity of $w$ and $\varphi$:

\begin{itemize}
\item[\tiny$\bullet$] The potential term $w \partial_x \varphi-\widehat{w\partial_x \varphi}(0)$ is harder to be bounded than $\partial_x (u^2)$ in the nonlinear KdV equation, which forces us to establish new bilinear estimates (see Lemma~\ref{Lemma-bilinear} and Corollary~\ref{Coro-bilinear}). For instance, if a sequence $u^n$ converges in $L^2(0,T;L^2)$, then it is immediate that $\partial_x (u^n)^2$ converges in $H^{-1}(0,T;L^1)$. On the other hand, $w\partial_x \varphi$ cannot be bounded by Sobolev norms in our setting.

\item[\tiny$\bullet$] The unique continuation property is also new due to lower regularities, to be compared with the prior works \cite{RZ-06,LRZ-10}  dealing with $w\in L^\infty(0,T;L^\infty)$ and $\varphi\in L^2(0,T;H^1)$. In the current situation, despite propagation of compactness allows us to improve the regularity of $\varphi$ to $L^2(0,T;H^{3/4+})$, it is still insufficient to carry out the reasoning from known results. We refer to Proposition~\ref{Prop-UCP} for further discussion.
\end{itemize}

\begin{proof}[{\bf Proof of Lemma~\ref{Lemma-fullobs}}]
One may assume the damping coefficient $a=0$; the setting $a>0$ can be translated into this case via $\varphi\mapsto e^{at}\varphi$, while $w$ is not changed and the conclusion is not affected owing to $e^{at}\in [1,e^{aT}]$. The proof is divided into three steps.

\medskip

\textit{Step 1 $($contradiction argument$)$.}
Assume that there exist sequences $\varphi^n_1$ and $w^n$ such that 
\begin{align}
& \|\varphi^n_1\|_{\dot{H}^{-s}}=1, \quad \|w^n\|_{\dot{Y}^{1/4+\sigma}_1}< R, \label{contradiction-2}\quad \int_{t_1}^{t_2}\|G\varphi^n(t)\|_{\dot{H}^{-s}(s_1,s_2)}^2dt\le \frac{1}{n}, \\
& \partial_t\varphi^n+\partial_x^3\varphi^n+f^n=0,\quad \varphi^n(1)=\varphi^n_1,\quad f^n:=w^n\partial_x\varphi^n-\widehat{w^n\partial_x\varphi^n}(0)\label{contradiction-1}.
\end{align}

Thanks to Proposition~\ref{Prop-linear-2}, Lemma~\ref{Lemma-bilinear} and \eqref{contradiction-2}, the sequence $\varphi^n$ is bounded in $\dot{Y}^{-s}_1\subset \dot{X}^{-s,1/2}_1$, and $f^n$ is bounded in $\dot{Z}^{-s}_1\subset \dot{X}^{-s,-1/2}_1$. 
In addition, let us fix a small constant
\[0<\varepsilon <\min \{(1/4+\sigma-s)/4,1/6,1/2-s\}.\]
Then the first inequality of Corollary~\ref{Coro-bilinear} (the parameter $s^*$ therein is taken as  $1/4+\sigma$ in the current setting) implies that $f^n$ is also bounded in $X^{-s-2\varepsilon,-1/2+\varepsilon}_1$; the latter space is compactly embedded in $X^{-s-3\varepsilon,-1/2}_1$, by Lemma~\ref{Lemma-Bourgain}(1).
		
Summarizing the above, there exists a subsequence (which we still denote with $n$), 
\begin{align}
w^n\rightarrow w\quad &{\rm weakly\ in\ }\dot{X}^{1/4+\sigma,1/2}_1 \ {\rm and\ strongly\ in\ }\dot{X}^{1/4+\sigma-\varepsilon,1/2-\varepsilon}_1 ,\label{convergence-1}\\
\varphi_1^n\rightarrow  \varphi_1\quad &{\rm weakly\ in\ }\dot{H}^{-s}, \notag\\
\varphi^n\rightarrow \varphi\quad &{\rm weakly\ in\ }\dot{X}^{-s,1/2}_1 \ {\rm and\ strongly\ in\ }\dot{X}^{-s-\varepsilon,1/2-\varepsilon}_1, \label{convergence-2}\\
f^n\rightarrow f\quad &{\rm weakly\ in\ }\dot{X}^{-s,-1/2}_1\ {\rm and\ strongly\ in\ }\dot{X}^{-s-3\varepsilon,-1/2}_1, \label{convergence-3}
\end{align}
as well as
\begin{equation}\label{convergence-4}
\partial_t\varphi+\partial_x^3\varphi+f=0,\quad \varphi(1)=\varphi_1.
\end{equation}
We point out that this equation holds in the sense of Duhamel formula (where the equality holds in $X^{-s,-1/2}_1$), according to the fact that for linear operators, the continuity in strong and weak topologies are equivalent. In particular, the estimates in Bourgain spaces, namely Lemma~\ref{lem_Bougainproperty}(1)-(3), yield that $\varphi\in \dot{Y}^{-s}_1$, and therefore $\varphi$ is a mild solution of \eqref{convergence-4}.

Although multiplication does not preserve weak convergence, we can still verify that
\begin{equation}\label{contradiction-8}
f=w\partial_x\varphi-\widehat{w\partial_x\varphi}(0).
\end{equation}
Indeed, in view of the strong convergences from \eqref{convergence-1} and \eqref{convergence-2}, as well as the second estimate in Corollary~\ref{Coro-bilinear} (with $s^*$ and $-s$ taken as $1/4+\sigma-\varepsilon$ and $-s-\varepsilon$, respectively), we obtain
\begin{equation}\label{contradiction-9}
f^n \to w \partial_x \varphi-\widehat{w \partial_x \varphi}(0) \quad {\rm strongly\ in\ }X^{-s-3\varepsilon,-1/2}_1.
\end{equation}
This together with \eqref{convergence-3} immediately implies \eqref{contradiction-8}.

\medskip
		
\textit{Step 2 $($strong convergence$)$.} This step aims to show
\begin{equation}\label{contradiction-6}
\varphi^n\rightarrow \varphi\quad {\rm strongly\ in\ }L^2_{loc}(t_1,t_2;H^{-s}(\T)).
\end{equation}
To this end, we first claim that 
\begin{equation}\label{contradiction-4}
G\varphi\equiv 0.
\end{equation}
Indeed, in view of \eqref{convergence-2} one gets that $\varphi^n\rightarrow \varphi$ weakly in $L^2(0,1;H^{-s})$. Notice that the multiplication by $\chi$ is bounded on $L^2 (0,1;H^s(\T))$. Accordingly,
\[G\varphi^n=\chi\left[\varphi^n-\tfrac{2\pi}{|s_1-s_2|}\widehat{\mathbf{1}_{(s_1,s_2)}\varphi^n}(0)\right]\rightarrow G\varphi\quad \text{weakly in }L^2(0,1;H^{-s}).\]
This together with \eqref{contradiction-2} leads to \eqref{contradiction-4}, as desired. 

To continue, we let $\psi^n=\varphi^n-\varphi$, and observe that by \eqref{contradiction-2} and \eqref{contradiction-4},
\begin{equation*}
\int_{t_1}^{t_2} \|G\psi^n\|_{H^{-s}(\T)}^2 dt\lesssim \int_{t_1}^{t_2}\|G\psi^n\|^2_{\dot{H}^{-s}(s_1,s_2)}dt=\int_{t_1}^{t_2}\|G\varphi^n\|^2_{\dot{H}^{-s}(s_1,s_2)}dt\to 0
\end{equation*}
(here we tacitly use the fact that $\dot{H}^{-s}(s_1,s_2)$ norm and $\dot{H}^{-s}(\T)$ norm are equivalent for functions supported in $(s_1,s_2)$).
At the same time, by \eqref{convergence-2} we have $\psi^n\rightarrow 0$ strongly in $L^2(0,1;H^{-s-\varepsilon})$. Thanks to $s+\varepsilon<1/2$, this yields that
\[\int_0^1|\widehat{\boldsymbol{1}_{(s_1,s_2)}\psi^n}(0)|^2dt\lesssim\int_0^1|\langle \boldsymbol{1}_{(s_1,s_2)},\psi^n\rangle_{H^{s+\varepsilon},H^{-s-\varepsilon}}|^2dt\lesssim \|\boldsymbol{1}_{(s_1,s_2)}\|_{H^{s+\varepsilon}}^2\int_0^1 \|\psi^n\|_{H^{-s-\varepsilon}}^2dt\rightarrow 0.\]
We thus derive from the last two inequalities (and the definition of $G$) that
\[\int_{t_1}^{t_2}\|\boldsymbol{1}_{(s_1,s_2)}\psi^n\|^2_{H^{-s}(\T)}dt\rightarrow 0.\]

Let us fix $(\tilde s_1,\tilde s_2) \Subset (s_1,s_2)$, and choose a cut-off function $\tilde \chi\in C^\infty(\T)$ satisfying 
\[\tilde\chi(x)=1,\quad x\in (\tilde s_1,\tilde s_2)\quad {\rm and }\quad \tilde\chi(x)=0,\quad x\notin (s_1,s_2).\]
It then follows that 
\begin{equation}\label{contradiction-7}
\int_{t_1}^{t_2}\|\tilde\chi\psi^n\|^2_{H^{-s}}dt=\int_{t_1}^{t_2}\|\tilde\chi\boldsymbol{1}_{(s_1,s_2)}\psi^n\|^2_{H^{-s}}dt\leq C\int_{t_1}^{t_2}\|\boldsymbol{1}_{(s_1,s_2)}\psi^n\|^2_{H^{-s}}dt\rightarrow 0.
\end{equation}

For $u\colon\T\rightarrow\R$, we define $D^{-s}u$ by $\widehat{D^{-s}u}(k)={\rm sgn}(k)|k|^{-s}\widehat u(k)$ for $k\neq 0$ and $\widehat{D^{-s}u}(0)=\widehat u(0)$. Then, let us introduce
\[\Psi^n=D^{-s}\psi^n,\quad  F^n=D^{-s}(f^n-f).\]
It follows from \eqref{contradiction-1} and \eqref{convergence-2}-\eqref{convergence-4} that
\begin{align*}
& \|\Psi^n\|_{\dot{X}_1^{0,1/2}}\leq C,\quad \Psi^n\rightarrow 0\quad {\rm strongly\ in\ }\dot{X}^{-\varepsilon,1/2-\varepsilon}_1,\\
& F^n\rightarrow 0\quad {\rm strongly\ in\ }\dot{X}^{-3\varepsilon,-1/2}_1
\end{align*}
and (in the sense of mild solution)
\[\partial_t\Psi^n+\partial_x^3\Psi^n+F^n=0.\]
At the same time, letting $[A,B]=AB-BA$ the commutator of linear operators, we notice that
\[\tilde \chi\Psi^n=[\tilde\chi,D^{-s}](\varphi^n-\varphi)+D^{-s}[\tilde\chi (\varphi^n-\varphi)].\]
This together with \eqref{convergence-2}, \eqref{contradiction-7} and $[\tilde\chi,D^{-s}]\in\mathcal L(H^{-s-1};L^2)$ (see, e.g.~\cite[Lemma A.1]{Laurent-ECOCV}) indicates that $\tilde \chi\Psi^n\rightarrow 0$ strongly in $L^2(t_1,t_2;L^2(\T))$, and thus
\[\Psi^n\rightarrow 0\quad {\rm strongly\ in\ }L^2(t_1,t_2;L^2(\tilde s_1,\tilde s_2)).\]
In conclusion, an application of Proposition~\ref{Prop-propagation-1}, with $b=1/2$, $b'=0$, $\omega=(\tilde s_1,\tilde s_2)$ and $(0,T)$ replaced by $(t_1,t_2)$, yields that
\[\Psi^n\rightarrow 0\quad {\rm strongly\ in\ }L^2_{loc}(t_1,t_2;L^2(\T)).\]
This is exactly the desired convergence \eqref{contradiction-6}.

\medskip
		
\textit{Step 3 $($completing the proof$)$.} Recalling identity \eqref{contradiction-4}, we derive that 
\begin{equation}\label{contradiction-5}
\varphi(t,x)\equiv c(t)\quad {\rm on\ }(t_1,t_2)\times (s_1,s_2).
\end{equation}
for some time-dependent function $c(t)$. Let us fix $(\tilde{t}_1,\tilde{t}_2)\Subset (t_1,t_2)$ and claim that $c\in H^{1/2}([\tilde{t}_1,\tilde{t}_2])$ (as a result, $(x,t)\mapsto c(t)$ belongs to $X^{r,1/2}_{[\tilde{t}_1,\tilde{t}_2]}$ for any $r\in \R$). Indeed, take any $g\in C_0^\infty((s_1,s_2))$, then Lemma~\ref{Lemma-Bourgain}(2) implies $g(x)\varphi(t,x)\in X^{-s-1,1/2}_{1}$; since $g(x)\varphi(t,x)= g(x)c(t)$ on $(t_1,t_2)\times \T$, we get $g(x)c(t)\in X^{-s-1,1/2}_{[\tilde{t}_1,\tilde{t}_2]}$, and thus the claim easily follows.

Combining \eqref{convergence-4} and \eqref{contradiction-8}, we find that 
\begin{equation}\label{equation-1}
\partial_t\varphi+\partial_x^3\varphi+w\partial_x\varphi-\widehat{w\partial_x\varphi}(0)=0,\quad t\in [\tilde{t}_1,\tilde{t}_2].
\end{equation}
Meanwhile, it follows from  
\eqref{contradiction-5} and \eqref{equation-1} that 
\[c'(t)=\widehat{w\partial_x \varphi}(0),\quad t\in(t_1,t_2).\]
Consider the auxiliary function $\tilde\varphi(t,x):=\varphi(t,x)-c(t)\in X^{-s,1/2}_1$, which satisfies
\begin{equation}\label{equation-2}
\partial_t \tilde\varphi+\partial_x^3 \tilde\varphi+w\partial_x \tilde\varphi=0,\quad t\in [\tilde{t}_1,\tilde{t}_2]
\end{equation}
and
\[\tilde\varphi(t,x)\equiv 0,\quad \forall (t,x)\in  [\tilde{t}_1,\tilde{t}_2]\times (s_1,s_2).\]

Let us proceed to increase the (spatial) regularity of $\tilde\varphi$ from $X^{-s,1/2}_{[\tilde{t}_1,\tilde{t}_2]}\subset L^2([\tilde{t}_1,\tilde{t}_2],H^{-s}(\T))$ to 
\begin{equation}\label{H3/4+sigma-}
    \tilde\varphi\in L^2_{loc} (\tilde{t}_1,\tilde{t}_2; H^{3/4+\sigma-}(\T)).
\end{equation}
To this end, notice that by Lemma~\ref{Lemma-bilinear},
\[w\partial_x \tilde\varphi=w\partial_x \varphi=w\partial_x\varphi-\widehat{w\partial_x\varphi}(0)+c'(t)=f+c'(t)\in X^{-s,-1/2}_{[\tilde{t}_1,\tilde{t}_2]}.\]
Therefore, an application of propagation of regularity (Proposition~\ref{Prop-propagation-2}, with $r=-s$, $b=1/2$ and $[0,T]$ replaced by $[\tilde{t}_1,\tilde{t}_2]$) yields that $\tilde\varphi\in L^2_{loc}(\tilde{t}_1,\tilde{t}_2;H^{-s+1/2})$, and therefore $\varphi\in L^2_{loc} (\tilde{t}_1,\tilde{t}_2;\dot{H}^{-s+1/2})$ as well. As pointed out in Step 1, $\varphi$ is a mild solution of \eqref{equation-1}, and belongs to $\dot{Y}^{-s}_{[\tilde{t}_1,\tilde{t}_2]}$. We can pick $t'\in (\tilde{t}_1,\tilde{t}_2)$ so that $\varphi(t')\in \dot{H}^{-s+1/2}$, the uniqueness of solution (Proposition~\ref{Prop-linear-1}) implies that $\varphi\in \dot{Y}^{-s+1/2}_{[\tilde{t}_1,\tilde{t}_2]}$. In particular, we have lifted the regularity of $\tilde{\varphi}$ to $X^{-s+1/2,1/2}_{[\tilde{t}_1,\tilde{t}_2]}$. Repeating the above procedure, one can arrive at $\tilde\varphi\in X^{1/4+\sigma-,1/2}_{[\tilde{t}_1,\tilde{t}_2]}$ (which is the maximal regularity for Proposition~\ref{Prop-linear-1} to hold, since $w\in \dot{Y}^{1/4+\sigma}_{[\tilde{t}_1,\tilde{t}_2]}$). We apply propagation of regularity once more to obtain \eqref{H3/4+sigma-}.

Finally, we need a ``unique continuation property'', stated in Proposition~\ref{Prop-UCP} below, which directly implies that $\tilde\varphi(t,x)\equiv 0$, namely $\varphi(t,x)\equiv c(t)$,  on $[\tilde{t}_1,\tilde{t}_2]\times \T$. Since $\varphi(t,\cdot)$ is of mean-zero, one has $\varphi(t,x)\equiv 0$ on $[\tilde{t}_1,\tilde{t}_2]\times\T$. Recalling \eqref{contradiction-6} one can take $t_0\in (\tilde{t}_1,\tilde{t}_2)$ satisfying
\[\varphi^n(t_0)\rightarrow \varphi(t_0)=0\quad {\rm strongly\ in\ }H^{-s}.\]
This, combined with \eqref{contradiction-2}, implies that 
\[1=\|\varphi_1^n\|_{H^{-s}}\leq C\|\varphi^n(t_0)\|_{H^{-s}}\rightarrow 0,\]
which is a contradiction. The proof is then complete.
\end{proof}

We supply Proposition~\ref{Prop-UCP} below for the unique continuation result required in Step 3 for $H^{3/4+\delta}$ solution of \eqref{equation-2}\footnote{The extra $\delta>0$ is in fact unnecessary. However, we are content with this slightly weaker formulation, to avoid some endpoint cases in Sobolev product estimates \eqref{productSobolev}.}. We also refer the reader to \cite[Corollary 3.8]{LRZ-10} for unique continuation of nonlinear KdV, and \cite[Lemma 3.5]{RZ-06} for unique continuation when the potential is $L^\infty$ and the solution is $H^1$. Under the assumptions of $w\in H^{1/4+\delta}$ and $u\in H^{3/4+\delta}$, the potential term $f=w\partial_x u$ belongs to $H^{-1/2}$. Our new observation is that, if
\[\partial_t u+\partial_x^3 u+f=0\]
and $u$ vanishes on a fixed non-empty open set $\omega\subset \T$, then $u$ can gain {\it 2} derivatives from $f$ (this is better than the $1/2$ derivative gain by propagation of regularity, and the $1$ derivative gain by Kato smoothing on the real line). Consequently, we find that $u\in H^{3/2}$, which compensates for $w\not\in L^\infty$ and allows us to modify the arguments in \cite{RZ-06} and finish the proof.

\begin{proposition}[Unique continuation]\label{Prop-UCP}
Let $T>0$, $\delta>0$, and a non-empty open set $\omega\subset \T$ be arbitrarily given. If $w\in L^\infty(0,T; H^{1/4+\delta})$ and $u\in L^2(0,T; H^{3/4+\delta})$ satisfy (in the sense of distributions)
\[\begin{cases}
\partial_t u+\partial_x^3 u+w\partial_x u=0,\quad x\in\T,\\
u(t,x)\equiv 0\quad \text{on }[0,T]\times\omega,
\end{cases}\]
then $u(t,x)\equiv 0$ on $[0,T]\times\T$.
\end{proposition}

\begin{proof}[{\bf Proof of Proposition~\ref{Prop-UCP}}]
The proof consists of two steps.

\medskip
		
\textit{Step 1 $($regularity of $u$$)$.} We show that the regularity of $u$ can be improved to $L^2_{loc}(0,T;H^{3/2})$. First recall a special case of the well-known product estimate in Sobolev spaces: if $s_1+s_2\ge 0$ and $s<\min\{s_1,s_2,s_1+s_2-1/2\}$, then for any $g,h\colon \T\to \R$, we have
\begin{equation}\label{productSobolev}
\|gh\|_{H^{s}(\T)}\lesssim \|g\|_{H^{s_1}(\T)} \|h\|_{H^{s_2}(\T)}.
\end{equation}
In addition, for any $s\in (0,1)$, the $\dot H^s (\T)$ norm of mean zero $g\colon \T\to \R$ is equivalent to
\begin{equation}\label{eq_H1/4equivalent}
\|g\|_{\dot H^s}=\left(\int_0^{\pi/2} \frac{\|D_h g\|_{L^2}^2}{h^{1+2s}}\, dh\right)^{1/2},
\end{equation}
where $D_h g:=g(\cdot+h)-g$ is the difference of $g$.

Fix any $\xi\in C_0^\infty(\R)$ with $\supp(\xi)\subset (0,T)$, then it suffices to address that $v:=\xi(t) u$ belongs to $L^2(\R;H^{3/2})$. Note that  $\|v\|_{L^2(\R;H^{3/4+\delta})}\le C\|u\|_{L^2(0,T;H^{3/4+\delta})}$ (the constants hereinafter may depend on $\xi$ and the open set $\omega$), and $v$ satisfies the equation
\begin{equation}\label{airy+2}
\partial_t v+\partial_x^3 v+f=0,\quad \text{where }f:=\xi(t) (w\partial_x u)-\xi'(t) u.
\end{equation}
Thanks to \eqref{productSobolev}, we find
\begin{equation}\label{fH-1/2}
\begin{aligned}
\|f\|_{L^2(\R;H^{-1/2})}&\le C \left(\|w\partial_x u\|_{L^2(0,T;H^{-1/2})}+\|u\|_{L^2(0,T;H^{-1/2})}\right)\\
&\le C\left(\|w\|_{L^\infty(0,T;H^{1/4+\delta})} \|u\|_{L^2(0,T;H^{3/4+\delta})}+\|u\|_{L^2(0,T;H^{-1/2})}\right)<\infty.
\end{aligned}
\end{equation}
We claim that $v$ gains two derivatives from $f$, namely
\begin{equation}\label{eq_H3/2estimate}
\|v\|_{L^2(\R;{H}^{3/2})}\le C\left(\|f\|_{L^2(\R;H^{-1/2})}+\|v\|_{L^2(\R;H^{1/2})}\right)<\infty.
\end{equation}

To prove the claim, we can assume that $v$ and $f$ are smooth, since convolution with a mollifier does not ruin the condition that $v=0$ on $\R\times \omega$, up to replacing $\omega$ with a smaller open set. In addition, let us identify $\T$ with $[0,2\pi]$ modulo $0\sim 2\pi$, and assume without loss of generality
\begin{equation}\label{itervalomega}
\omega=[0,\varepsilon)\cup (2\pi-\varepsilon,2\pi].
\end{equation}
Let us multiply \eqref{airy+2} by $x u$ and integrate with respect to $x\in [0,2\pi]$. Note that $v(t,x)$ vanishes near $x=0$ and $x=2\pi$ for any $t\in \R$, which leads to
\[\int_0^{2\pi} xv \partial_x^3 v dx=-\int_0^{2\pi} (v+x\partial_x v)\partial_x^2 v dx=\|\partial_x v\|_{L^2}^2-\frac{1}{2}\int_0^{2\pi}   x \frac{d}{dx}(\partial_x v)^2 dx=\frac{3}{2} \|\partial_x v\|_{L^2}^2.\]
Therefore
\begin{equation}\label{airy-xuinequalti}
    \frac{1}{2}\frac{d}{dt}\int_0^{2\pi} xv^2 dx+\frac{3}{2}\|\partial_x v\|_{L^2}^2=-\int_0^{2\pi} x v f dx.
\end{equation}
The right-hand side can be bounded by
\begin{align*}
    \left|\int_0^{2\pi} x v f dx\right|&\le  \|xv\|_{H^1}\|f\|_{H^{-1}}\le C\|v\|_{H^1}\|f\|_{H^{-1}} \le C \left(\|v\|_{L^2}+\|\partial_x v\|_{L^2}\right)\|f\|_{H^{-1}}\\
    &\le \frac{1}{2}\|\partial_x v\|_{L^2}^2+C\left(\|f\|_{H^{-1}}^2+\|v\|_{L^2}^2\right).
\end{align*}
Here the estimate $\|xv\|_{H^1}\le C\|v\|_{H^1}$ may need some clarification, since by $H^1$ we mean the Sobolev space on $\T$, and $x$ is not a smooth function on $\T$. Indeed, we can fix any $\zeta\in C^\infty(\T)$ so that $\zeta$ coincides with $x$ on $(\varepsilon/2,2\pi-\varepsilon/2)$. 
This, together with the condition $v(t,x)=0$ for $x\in \omega$, indicates
$x v=\zeta v$. Thus we can safely use the product estimate on $\T$. The first term can be absorbed into the left-hand side of \eqref{airy-xuinequalti}, which yields
\[\frac{1}{2}\frac{d}{dt}\int_0^{2\pi} xv^2 dx+\|\partial_x v\|_{L^2}^2\le C\left(\|f\|_{H^{-1}}^2+\|v\|_{L^2}^2\right).\]
Integrating with respect to $t\in \R$, the first term vanishes as $v(t)=0$ for $t\not \in (0,T)$, and thus
\begin{equation}\label{eq_vH1estimate}
\|\partial_x v\|_{L^2(\R;L^2)}^2\le C\left(\| f\|_{L^2(\R;H^{-1})}^2+\|v\|_{L^2(\R;L^2)}^2\right).
\end{equation}

To recover the extra $1/2$ derivative in \eqref{eq_H3/2estimate}, let us replace $v$ by its difference $D_h v$ for $h\in (0,\pi)$, and note that $D_h$ commutes with derivatives. Then using the equivalent norm \eqref{eq_H1/4equivalent} leads to
\[\|v\|_{L^2(\R;{H}^{3/2})}^2\lesssim \|v\|_{L^2(\R;L^2)}^2+\|\partial_x v\|_{L^2(\R;\dot{H}^{1/2})}^2\lesssim \|v\|_{L^2(\R;L^2)}^2+\int_0^{\pi/2} \frac{\|\partial_x D_h v\|_{L^2(\R;L^2)}^2}{h^2}\, dh.\]
The first term is clearly bounded by the right-hand side of \eqref{eq_H3/2estimate}. For the second term, when $h\ge \varepsilon/2$, we use the trivial bound
\[\|\partial_x D_h v\|_{L^2(\R;L^2)}\le \|\partial_x v(\cdot+h)\|_{L^2(\R;L^2)}+\|\partial_x v\|_{L^2(\R;L^2)}\le 2\|\partial_x v\|_{L^2(\R;L^2)}.\]
And when $h<\varepsilon/2$, the difference $D_h v$ satisfies
\[\partial_t D_h v+\partial_x^3 D_h v+D_h f=0.\]
Since $D_h v$ vanishes in $[0,\varepsilon/2)\cup (2\pi-\varepsilon/2,2\pi]$, the inequality \eqref{eq_vH1estimate} is still valid with $v,f$ replaced by $D_hv, D_hf$. Consequently, with implicit constants depending on $\varepsilon$, we have
\begin{align*}
\int_0^{\pi/2} \frac{\|\partial_x D_h v\|_{L^2(\R;L^2)}^2}{h^2}\, dh&\lesssim \int_0^{\varepsilon/2} \frac{\|(1-\Delta)^{-1/2} D_h f\|_{L^2(\R;L^2)}^2+\|D_h v\|_{L^2(\R;L^2)}^2}{h^2} dh+\|\partial_x v\|_{L^2(\R;L^2)}^2\\
&\lesssim \|(1-\Delta)^{-1/2} f\|_{L^2(\R;{H}^{1/2})}^2+\| v\|_{L^2(\R;{H}^{1/2})}^2+\|\partial_xv\|_{L^2(\R;L^2)}^2\\
&\lesssim \|f\|_{L^2(\R;H^{-1/2})}^2+\|v\|_{L^2(\R;H^{1/2})}^2+\|\partial_xv\|_{L^2(\R;L^2)}^2.
\end{align*}
In view of \eqref{fH-1/2} and \eqref{eq_vH1estimate}, we arrive at the claim \eqref{eq_H3/2estimate}. Specifically, the arbitrariness of $\xi$ implies $u\in L^2_{loc}(0,T;H^{3/2})$.

\medskip
		
\textit{Step 2 $($apply Carleman estimate$)$.} Since we can replace $[0,T]$ by any sub-interval, we can assume $u\in L^2(0,T;H^{3/2})$ without loss of generality. The rest of the proof only relies on a weaker condition
\[u\in L^2(0,T;H^{5/4}).\]
As a result, making use of \eqref{productSobolev} again we improve the estimate for $w\partial_x u$ in \eqref{fH-1/2}:
\begin{equation}\label{eq_L2tx}
\|w\partial_x u\|_{L^2(0,T;L^2)}\le C\|w\|_{L^\infty(0,T;H^{1/4+\delta})}\|u\|_{L^2(0,T;H^{5/4})}<\infty.
\end{equation}

In the sequel we follow the idea of \cite{RZ-06} to consider the integral in time and further raise the spatial regularity to $H^3$, in order to meet the assumptions for a Carleman estimate. Indeed, fix any $T'\in(0,T)$ and a small parameter $h\in(0,T-T')$, and set
\[u^h(t,x)=\frac{1}{h}\int_t^{t+h}u(r,x)dr,\quad t\in[0,T'].\]
Then it follows that $u^h\in H^1(0,T';H^{5/4})$ and 
\begin{equation*}
\begin{cases}
\partial_t u^h+\partial_x^3u^h+[w\partial_xu]^h=0,\quad x\in\T,\\
u^h(t,x)= 0\quad \text{on }[0,T']\times \omega.
\end{cases}
\end{equation*}
In view of $\partial_t u^h\in L^2(0,T';H^{5/4})$ and \eqref{eq_L2tx}, we find $\partial_x^3u^h\in L^2(0,T';L^2)$ and thus
\[u^h\in L^2(0,T';H^3).\]
In addition, by \eqref{itervalomega}, we still have 
\[\partial_x^j u^h(t,0)=\partial_x^j u^h(t,2\pi)=0,\quad j=0,1,2.\]
As a consequence, the function $u^h$ verifies the assumptions of Carleman estimate in \cite[Lemma 3.3]{RZ-06}), which implies that there exists a smooth positive function $\psi$ on $[0,2\pi]$ satisfying 
\begin{align}
&\int_0^{T'}\int_0^{2\pi}\left[
\frac{\lambda^5}{t^5(T'-t)^5}|u^h|^2+\frac{\lambda^3}{t^3(T'-t)^3}|\partial_xu^h|^2+\frac{\lambda}{t(T'-t)}|\partial_x^2u^h|^2
\right]\exp\left\{
-\frac{2\lambda\psi(x)}{t(T'-t)}
\right\}dxdt\notag\\
\lesssim &\int_0^{T'}\int_0^{2\pi} |\partial_t u^h+\partial_x^3 u^h|^2\exp\left\{
-\frac{2\lambda\psi(x)}{t(T'-t)}
\right\}dxdt\label{Carleman-estimate}\\
=&\int_0^{T'}\int_0^{2\pi} |[w\partial_xu]^h|^2\exp\left\{
-\frac{2\lambda\psi(x)}{t(T'-t)}
\right\}dxdt\notag
\end{align}
for any sufficiently large $\lambda>0$, where the implicit constant in $\lesssim$ does not depend on $\lambda$ and $u^h$. We split the right-hand side into two terms
\begin{equation}\label{Carleman-modify}
\begin{aligned}
    \text{RHS of }\eqref{Carleman-estimate}&\lesssim \int_0^{T'}\int_0^{2\pi}|w\partial_xu^h|^2\exp\left\{
-\frac{2\lambda\psi(x)}{t(T'-t)}
\right\}dxdt\\
&\quad + \int_0^{T'}\int_0^{2\pi}|[w\partial_xu]^h-w\partial_xu^h|^2\exp\left\{
-\frac{2\lambda\psi(x)}{t(T'-t)}
\right\}dxdt.
\end{aligned}
\end{equation}

Denoting $\vartheta(t,x)=\exp\left\{-\frac{\lambda\psi(x)}{t(T'-t)}\right\}$ for simplicity, one can check that
\begin{equation*}
\begin{aligned}
\|w\partial_xu^h\vartheta\|_{L^2}^2&\leq \|w\|_{L^2}^2\|\partial_xu^h\vartheta\|_{L^\infty}^2\lesssim \|w\|_{L^2}^2\|\partial_xu^h\vartheta\|_{L^2}\|\partial_xu^h\vartheta\|_{H^1}\\
&\lesssim \|w\|_{L^2}^2\left(\|\partial_xu^h\vartheta\|_{L^2}^2+\|\partial_xu^h\vartheta\|_{L^2}\|\partial_x^2u^h\vartheta\|_{L^2}+\|\partial_xu^h\partial_x\vartheta\|_{L^2}^2
\right).
\end{aligned}
\end{equation*}
This together with the fact $w\in L^\infty(0,T;L^2(\T))$ indicates that
\begin{equation*}
\begin{aligned}
&\int_0^{T'}\int_0^{2\pi}|w\partial_xu^h|^2\exp\left\{
-\frac{2\lambda\psi(x)}{t(T'-t)}
\right\}dtdx\\
&\lesssim \int_0^{T'}\int_0^{2\pi}\left(|\partial_xu^h|^2+\frac{\lambda^2}{t^2(T'-t)^2}|\partial_x u^h|^2+|\partial_x^2u^h|^2\right)
\exp\left\{
-\frac{2\lambda\psi(x)}{t(T'-t)}
\right\}dtdx.
\end{aligned}
\end{equation*}
Therefore, 
the first integral in \eqref{Carleman-modify} can be absorbed into the left-hand side of \eqref{Carleman-estimate} by increasing $\lambda$. At the same time, owing to the estimate \eqref{eq_L2tx}, when $h\rightarrow 0^+$ we have,
\begin{equation*}
\begin{aligned}
&[w\partial_xu]^h\rightarrow w\partial_xu \quad \text{in }L^2(0,T';L^2(\T)),\\
&w\partial_xu^h\rightarrow w\partial_xu \quad \text{in }L^2(0,T';L^2(\T)).
\end{aligned}
\end{equation*}
This indicates that the second integral in \eqref{Carleman-modify} converges to $0$. We thus conclude that 
\[\int_0^{T'}\int_0^{2\pi} 
\frac{\lambda^5}{t^5(T'-t)^5}|u^h|^2\exp\left\{
-\frac{2\lambda\psi(x)}{t(T'-t)}
\right\}dxdt\rightarrow 0.\]
This together with the fact $u^h\rightarrow u$ in $L^2(0,T';L^2)$ gives $u(t,x)=0$ on $(0,T')\times \T$. Finally, the conclusion of this proposition follows from the arbitrariness of $T'$.
\end{proof}

Finally, taking the HUM arguments in Appendix~\ref{Section-HUM} for granted, now the low frequency controllability (Proposition~\ref{Prop-LFcontrol}) has been established.

\subsection{High-frequency dissipation via nonlinear smoothing}\label{Section-HF}

The following result serves to deal with the high frequencies of controlled system \eqref{Control-problem}, which can be viewed as an extension of ``nonlinear smoothing'' (Proposition~\ref{prop_nonlinearsmoothing}) to linearized equation.

\begin{proposition}\label{Prop-linearized_nonlinearsmooothing}
Given $\sigma\in (0,1/4)$ and $R>0$, if $w\in \dot{Y}^{1/4+\sigma}_1$ is the solution of uncontrolled system \eqref{Uncontrolled-Problem} with initial data $w_0$ and source term $h$ satisfying
\[\|w_0\|_{\dot{H}^{1/4+\sigma}}<R,\quad \|h\|_{L^2(0,1;\dot{H}^{1/4+\sigma})}<R,\]
then for any $v_0\in \dot{L}^2$ and $f\in L^2(0,1;\dot{H}^{\sigma})$, the solution $v=\mathcal{V}_w(v_0,f)$ of linearized system \eqref{Control-problem} (with $f$ in place of $\chi \mathcal{P}_N\xi$ on RHS) satisfies
\[\|v(t)-S_a(t)v_0\|_{\dot{H}^{\sigma}}\leq C(\sigma,R)\left(
\|v_0\|_{L^2}+\|f\|_{L^2(0,1;\dot{H}^{\sigma})} \right),\quad \forall t\in [0,1].\]
\end{proposition}

An immediate corollary is the high frequency dissipation, since $\|Id-P_m\|_{\mathcal L(\dot{H}^{\sigma}, \dot{L}^2)}\le m^{-\sigma}$.

\begin{corollary}\label{Coro-HFdissipation}
    Under the assumptions of Proposition~{\rm\ref{Prop-linearized_nonlinearsmooothing}}, for any $m\in \N^+$,
    \[\|(Id-P_m) [v(1)-S_a(1) v_0]\|_{\dot{L}^2}\le C(\sigma,R)m^{-\sigma} (\|v_0\|_{L^2}+\|f\|_{L^2(0,1;\dot{H}^{\sigma})}).\]
\end{corollary}

Before proving this proposition, we present a technical lemma, indicating a smoothing effect for the bilinear operator $\mathcal{B}$ (defined in \eqref{Multilinear-operator}) supplementary to Lemma~\ref{Lemma-BID}.

\begin{lemma}\label{Lemma-smoothing-1}
Let $\sigma\in (0,1/4)$ be arbitrarily given. Then there exists a constant $C>0$ such that
\[\|\mathcal{B}(\partial_x(w^2), v)\|_{\dot{H}^{\sigma}}\le C\|w\|_{\dot{H}^{1/4+\sigma}}^2\|v\|_{\dot{L}^2}\]
for any $w\in\dot{H}^{1/4+\sigma}$ and $v\in \dot{L}^2$.
\end{lemma}

\begin{proof}[{\bf Proof of Lemma~\ref{Lemma-smoothing-1}}]
Note that if $\widehat{f}(0)=\widehat{g}(0)=0$, then
\[\mathcal{B}(f,g)=\frac{1}{6}\cdot \partial_x^{-1} f\cdot \partial_x^{-1} g.\]
Thus
\[\mathcal{B}(\partial_x (w^2),v)=\mathcal{B}(\partial_x (w^2-\widehat{w^2}(0)),v)=\frac{1}{6}\cdot (w^2-\widehat{w^2}(0))\cdot \partial_x^{-1} v.\]
Since multiplication by $H^1$ function is a bounded operator on $H^{\sigma}$, we have
\[\|\mathcal{B}(\partial_x (w^2),v)\|_{H^{\sigma}}\le C\|w^2-\widehat{w^2}(0)\|_{H^{\sigma}}\|\partial_x^{-1} v\|_{H^1}\le C\left(\|w^2\|_{\dot{H}^{\sigma}}+|\widehat{w^2}(0)|\right)\|v\|_{\dot{L}^2}.\]
Clearly $|\widehat{w^2}(0)|=\|w\|_{\dot{L}^2}^2$. 
In the product estimate \eqref{productSobolev}, let us take $g=h=w$, $s_1=s_2=1/4+\sigma$ and $s=\sigma$. This yields $\|w^2\|_{\dot{H}^{\sigma}}\lesssim \|w\|_{\dot{H}^{1/4+\sigma}}^2$. Now the proof is complete.
\end{proof}

\begin{proof}[{\bf Proof of Proposition~\ref{Prop-linearized_nonlinearsmooothing}}]
We first point out that $\|w\|_{\dot{Y}^{1/4+\sigma}_1}\leq C(R)$ by Proposition~\ref{Prop-uncontrolled}, and therefore, by Proposition~\ref{Prop-linear-1},
\[\|v\|_{\dot{Y}^0_1}\le C(R)\left(\|v_0\|_{\dot{L}^2}+\|f\|_{L^2(0,1;\dot{L}^2)}\right).\]

Recall the multi-linear operators $\mathcal{B},\mathcal I,\mathcal D$ defined by \eqref{Multilinear-operator}.
Then, making use of the normal form transformation (cf.~Lemma~\ref{Lemma-DBP}), we can rewrite the function $v$ in the form
\begin{equation}\label{identity-4}
\begin{aligned}
v(t)&=S_a(t) v_0+2S_a(t)\mathcal{B}(w(0),v_0)-2\mathcal{B}(w(t),v(t))\\
& \quad +\int_0^tS_a(t-r)\left[-2a\mathcal{B}(w,v)+2\mathcal{B}(w,f)+2\mathcal I(w,w,v)+2 \mathcal{D}(w,w,v)+f(r)\right]dr\\
&\quad +2\int_0^tS_a(t-r)\mathcal{B}(S_a(r)\partial_r(S_a(-r)w),v)dr.
\end{aligned}
\end{equation}
The verification of formula \eqref{identity-4} will be provided at the end of this subsection.

\vspace{2mm}

We proceed to deal with the last integral in the RHS of \eqref{identity-4}, while the others can be addressed by Lemma~\ref{Lemma-BID}. Since (note that $r$ represents the time variable)
\[S_a(r)\partial_r(S_a(-r)w)=\partial_r w+\partial_x^3 w+aw=-\frac{1}{2}\partial_x(w^2)+h,\]
applying Lemma~\ref{Lemma-BID}(1) and Lemma~\ref{Lemma-smoothing-1} gives us
\begin{equation}\label{eq-hf-1}
\begin{aligned}
&\|\mathcal{B}(S_a(r)\partial_r (S_a(-r)w),v)\|_{\dot{H}^{\sigma}}\lesssim \|\mathcal{B}(\partial_x (w^2), v)\|_{\dot{H}^{\sigma}}+\|\mathcal{B} (h,v)\|_{\dot{H}^{\sigma}}\\
\lesssim &\left(\|w\|_{\dot{H}^{1/4+\sigma}}^2+\|h\|_{\dot{L}^2}\right)\|v\|_{\dot{L}^2}\lesssim (\|w\|_{\dot{Y}^{1/4+\sigma}_1}^2+\|h\|_{\dot{L}^2})\|v\|_{\dot{Y}^0_1}.
\end{aligned}
\end{equation}

Finally, applying Lemma~\ref{lem_Bougainproperty}(1)(2), Lemma~\ref{Lemma-BID} and \eqref{eq-hf-1} to \eqref{identity-4}, we obtain that for any $t\in [0,1]$, (cf.~the proof of Proposition~\ref{prop_nonlinearsmoothing})
\begin{align*}
\|v(t)-S_a(t) v_0\|_{\dot{H}^{\sigma}}&\lesssim e^{-at}\|w(0)\|_{\dot{L}^2} \|v_0\|_{\dot{L}^2}+\|w(t)\|_{\dot{L}^2}\|v(t)\|_{\dot{L}^2}\\
&\quad +\int_0^t e^{-a(t-r)} \left(\|w\|_{\dot{L}^2}\|v\|_{\dot{L}^2}+\|w\|_{\dot{L}^2}\|f\|_{\dot{L}^2}+\|w\|_{\dot{H}^{\sigma}}\|w\|_{\dot{L}^2}\|v\|_{\dot{L}^2}\right) dr\\
&\quad +\left\|\int_0^t S_a(t-r) \mathcal{D}(w,w,v) dr\right\|_{\dot{X}^{\sigma,b}_1}+\left\|\int_0^t S_a(t-r) f(r) dr\right\|_{\dot{Y}^{\sigma}}\\
&\quad +\int_0^t e^{-a(t-r)} \left(\|w\|_{\dot{Y}^{1/4+\sigma}_1}^2+\|h\|_{\dot{L}^2}\right)\|v\|_{\dot{Y}^1_0}dr\\
&\lesssim \|w\|_{\dot{Y}^0_1}\|v\|_{\dot{Y}^0_1}+\|w\|_{\dot{Y}^0_1}\|f\|_{L^2(0,1;\dot{L}^2)}+\|w\|_{\dot{Y}^{\sigma}_1} \|w\|_{\dot{Y}^0_1} \|v\|_{\dot{Y}^0_1}\\
&\quad +\|\mathscr{D}(w,w,v)\|_{\dot{X}^{\sigma,b-1}_1}+\|f\|_{L^2(0,1;\dot{H}^{\sigma})}\\
&\quad +\left(\|w\|_{\dot{Y}^{1/4+\sigma}_1}^2+\|h\|_{L^2(0,1;\dot{L}^2)}\right)\|v\|_{\dot{Y}^0_1}\\
&\lesssim C(R)\left(\|v_0\|_{\dot{L}^2}+\|f\|_{L^2(0,1;\dot{H}^{\sigma})}\right)+\|w\|_{\dot{X}^{0,1/2}_1}^2 \|v\|_{\dot{X}^{0,1/2}_1}\\
&\lesssim C(R)\left(\|v_0\|_{\dot{L}^2}+\|f\|_{L^2(0,1;\dot{H}^{\sigma})}\right).
\end{align*}
Now the proof is complete.
\end{proof}

Note that the source term $f$ is assumed to be smoother than $L^2$ for nonlinear smoothing to be valid. However, the low-frequency controllability from the previous subsection merely ensures $\xi\in L^2(t_1,t_2;\dot{L}^2(s_1,s_2))$ when $v_0\in \dot{L}^2$\footnote{Even though the projection $\mathcal{P}_N$ implies $\chi \mathcal{P}_N \xi\in L^2(0,1;H^{1/2-\varepsilon})$, the problem is that its norm in this smoother space cannot be control by the $L^2(t_1,t_2;\dot{L}^2(s_1,s_2))$ norm of $\xi$ uniformly in $N$. In the sequel, it is essential that we have the estimate $\|\chi \mathcal{P}_N \xi\|_{L^2(0,1;\dot{H}^s)}\le C \|\xi\|_{L^2(t_1,t_2;\dot{H}^s(s_1,s_2))}$ with the constant $C$ independent of $N$.}. To improve the regularity of $\xi$, we invoke a trick from \cite{LWX-24,CXZZ-25}, at the cost of steering the low frequencies of $v(1)$ to the linear evolution $S_a(1)v_0$ (instead of $0$), which still possesses a decay rate $e^{-a}$.

\begin{lemma}\label{lem-LFtoS(t)v0}
Given $\sigma\in (0,1/4)$, $R>0$ and $m\in \N^+$, there exists a constant $N\in \N^+$ such that the following assertions hold.
\begin{enumerate}
\item[$(1)$] Let $w\in \dot{Y}^{1/4+\sigma}_1$, $w_0\in \dot{H}^{1/4+\sigma}$ and $h\in L^2(0,1;\dot{H}^{1/4+\sigma})$ satisfy the assumption in Theorem~{\rm\ref{Theorem-control}(1)}. Then for any $v_0\in \dot{L}^2(\T)$, there exists a control $\xi$ such that
\begin{equation}\label{eq_controlbound-LFtoS(t)v0}
\int_{t_1}^{t_2} \|\xi(t,\cdot)\|_{\dot{H}^{\sigma}(s_1,s_2)}^2 dt\le C(\sigma,R)\|v_0\|_{\dot{L}^2}^2,
\end{equation}
and the solution $v\in \dot{Y}^s_1$ of linearized system \eqref{Control-problem} satisfies
\begin{equation}\label{Squeezing-LFtoS(t)v0}
P_m [v(1)-S_a(1) v_0]=0.
\end{equation}

\item[$(2)$] The control $\xi$ in {\rm(1)} can be specified as in Theorem~{\rm\ref{Theorem-control}(2)}.
\end{enumerate}
\end{lemma}

\begin{proof}[\bf Proof of Lemma~\ref{lem-LFtoS(t)v0}]
Note that $\|w\|_{\dot{Y}^{1/4+\sigma}_1}\le C(R)$. We split the solution $v$ into three components. First, consider the uncontrolled system
\[\begin{cases}
v^1_t+\partial_x^3v^1+av^1+\partial_x(wv^1)=0,\\
v^1(0,\cdot)=v_0.
\end{cases}\]
According to Proposition~\ref{Prop-linearized_nonlinearsmooothing}, we have
\[\|v^1(1)-S_a(1)v_0\|_{\dot{H}^{\sigma}}\le C(R) \|v_0\|_{\dot{L}^2}.\]
Next, consider the backward system
\[\begin{cases}
v^2_t+\partial_x^{3}v^2+av^2+\partial_x (wv^2)=0,\\
v^2(1,\cdot)=-[v^1 (1)-S_a(1)v_0]\in \dot{H}^{\sigma}.
\end{cases}\]
The well-posedness is similar to the forward system (see Proposition~\ref{Prop-linear-1}) up to a change of variables $(t,x)\mapsto (1-t,-x)$. In particular, we have $v^2\in \dot{Y}^{\sigma}_T$ and the a priori estimate
\[\|v^2(0)\|_{\dot{H}^{\sigma}}\le C(R) \|v^2(1)\|_{\dot{H}^{\sigma}}\le C(R) \|v_0\|_{\dot{L}^2}.\]
And finally, let $v^3$ solve the controlled system
\[\begin{cases}
v^3_t+\partial_x^3v^3+av^3+\partial_x (wv^3)=\chi \mathcal{P}_N \xi,\\
v^3(0,\cdot)=-v^2(0,\cdot)\in \dot{H}^{\sigma}.
\end{cases}\]
According to Proposition~\ref{Prop-LFcontrol}, we can find $\xi\in L^2(t_1,t_2;\dot{H}^{\sigma}(s_1,s_2))$ satisfying \eqref{eq_controlbound-LFtoS(t)v0}, so that
\[P_m v^3(1)=0.\]
Now we set $v=v^1+v^2+v^3$, which coincides with $\mathcal{V}_w (v_0,\chi\mathcal{P}_N \xi)$. Then \eqref{Squeezing-LFtoS(t)v0} is valid since
\[P_m v(1)=P_m(S_a(1)v_0+v^3(1))=P_m S_a(1) v_0.\]
    
Finally, the assertion (2) on the structure of control is easy, thanks to Proposition~\ref{Prop-LFcontrol}(2), the fact that $\dot{H}^{\sigma}(s_1,s_2)\hookrightarrow \dot{L}^2(s_1,s_2)$ and $\dot{L}^2\ni v_0\mapsto -v^2(0,\cdot)\in \dot{H}^{\sigma}$ are bounded linear operators, and the locally Lipschitz and $C^1$ dependence on $w_0$ and $h$ of the solution maps (see Proposition~\ref{Prop-uncontrolled} and Proposition~\ref{Prop-linear-1}).
\end{proof}

We conclude this subsection with the derivation of normal form transformation \eqref{identity-4} for linearized KdV, which is analogous to \cite[Proposition~4.4]{ET-16}.

\begin{proof}[\bf Proof of identity \eqref{identity-4}]
    The Fourier coefficients of $v(t)$ satisfies
    \begin{equation*}
        \partial_t \widehat v(k)=(ik^3-a) \widehat v(k)-ik\sum_{k_1+k_2=k}\widehat{w}(k_1) \widehat{v}(k_2)+\widehat{f}(k).
    \end{equation*}
    We keep the convention that all the frequencies $k_j$ and $k$ are non-zero. Let us set 
\begin{equation}\label{Transformation}
W(t)=S_a(-t)w(t),\quad V(t)=S_a(-t)v(t),\quad F(t)=S_a(-t)f(t).
\end{equation}
Owing to the identity $k_1^3+k_2^3-(k_1+k_2)^3=-3k_1 k_2(k_1+k_2)=-3kk_1 k_2$, it follows that
\begin{equation}\label{identity-0}
\begin{aligned}
\partial_t\widehat{V}(k)&=-ike^{-at}\sum_{k_1+k_2=k}e^{-3ikk_1k_2t}\widehat{W}(k_1) \widehat{V}(k_2)+\widehat F(k)\\
&=\frac{1}{3}e^{-at}\sum_{k_1+k_2=k}\partial_t(e^{-3ikk_1k_2t})\frac{\widehat{W}(k_1) \widehat{V}(k_2)}{k_1k_2}+\widehat F(k),
\end{aligned}
\end{equation}
We introduce an auxiliary mean-zero bilinear operator $\mathcal{B}_1$, whose Fourier coefficients are
\begin{equation*}
\widehat{\mathcal{B}_1(f,g)}(k):=-\frac{1}{3}\sum_{k_1+k_2=k}e^{-3ikk_1k_2t}\frac{\widehat{f}(k_1)\widehat{g}(k_2)}{k_1k_2}.
\end{equation*}
Then \eqref{identity-0} implies
\begin{equation}\label{identity-1}
\begin{aligned}
\partial_t[\widehat{V}(k)+e^{-at}\widehat{\mathcal{B}_1(W,V)}(k)]
&=e^{-at}[-a\widehat{\mathcal{B}_1(W,V)}(k)\\
&\quad +\widehat{\mathcal{B}_1(\partial_tW,V)}(k)+\widehat{\mathcal{B}_1(W,\partial_t V)}(k)]+\widehat F(k)
\end{aligned}
\end{equation}
According to the first line of \eqref{identity-0}, the term $\widehat{\mathcal{B}_1(W,\partial_t V)}(k)$ can be represented as
\begin{equation*}
\begin{aligned}
\widehat{\mathcal{B}_1(W,\partial_t V)}(k)&=\frac{i}{3}e^{-at}\sum_{\substack{k_1+k_2+k_3=k\\
k_2+k_3\not =0}}e^{-3i[kk_1(k_2+k_3)+k_2k_3(k_2+k_3)]t}\frac{\widehat{W}(k_1)\widehat{W}(k_2)\widehat{V}(k_3)}{k_1}+\widehat{\mathcal{B}_1(W,F)}(k)\\
&=\frac{i}{3}e^{-at}\sum_{\substack{k_1+k_2+k_3=k\\
k_2+k_3\not =0}}e^{-3i(k_1+k_2)(k_2+k_3)(k_1+k_3)t}\frac{\widehat{W}(k_1)\widehat{W}(k_2)\widehat{V}(k_3)}{k_1}+\widehat{\mathcal{B}_1(W,F)}(k).
\end{aligned}
\end{equation*}
We proceed to deal with the frequencies satisfying $(k_1+k_2)(k_2+k_3)(k_1+k_3)=0$. When $k_1+k_2=0$ and $k_1+k_3\neq 0$, the corresponding summand is
\[\widehat{V}(k)\sum_{|j|\neq |k|}\frac{1}{j}|\widehat{W}(j)|^2=0.\]
When $k_1+k_2=0$ and $k_1+k_3=0$, the corresponding summand is
\[-\frac{1}{k}\widehat{W}(-k)\widehat{W}(k)\widehat{V}(k).\]
When $k_1+k_2\neq 0$ and $k_1+k_3=0$, the corresponding summand is
\[\widehat{W}(k)\sum_{|j|\neq |k|}\frac{1}{j}\widehat{W}(j)\widehat{V}(-j).\]
In conclusion, it follows that
\begin{equation}\label{identity-2}
\widehat{\mathcal{B}_1(W,\partial_t V)}(k)=e^{-at}\left[2\widehat{\mathcal I(W,W,V)}(k)+\widehat{\mathcal D_1(W,W,V)}(k)\right]+\widehat{\mathcal{B}_1(W,F)}(k).
\end{equation}
where $\mathcal{D}_1$ is a mean-zero trilinear operator defined by
\[\widehat{\mathcal D_1(f,g,h)}(k)=\frac{i}{3}\sum_{\substack{k_1+k_2+k_3=k\\(k_1+k_2)(k_2+k_3)(k_1+k_3)\not =0}}e^{-3i(k_1+k_2)(k_2+k_3)(k_1+k_3)t}\frac{\widehat{f}(k_1)\widehat{g}(k_2)\widehat{h}(k_3)}{k_1}.\]

Now, substituting \eqref{identity-2} into \eqref{identity-1} and comparing the Fourier coefficients, we get
\begin{equation*}
\begin{aligned}
\partial_t(U+e^{-at}\mathcal{B}_1(W,V))&=e^{-at}[-a\mathcal{B}_1(W,V)+\mathcal{B}_1(\partial_t W,V)+\mathcal{B}_1(W,F)] \\
&\quad +e^{-2at}\left[2\mathcal I(W,W,V)+\mathcal D_1(W,W,V)\right]+F.
\end{aligned}
\end{equation*}
Integrate this identity on $[0,t]$. By comparing the Fourier coefficients, it is easy to see
\begin{align*}
&\mathcal{B}_1(S_a(-t)f,S_a(-t)g)=2e^{at} S_a(-t)\mathcal{B}(f,g),\\
&\mathcal I(S_a(-t)f,S_a(-t)g,S_a(-t)h)=e^{2at} S_a(-t)\mathcal{I}(f,g,h),\\
&\mathcal D_1(S_a(-t)f,S_a(-t)g,S_a(-t)h)=2e^{2at} S_a(-t) \mathcal{D}(f,g,h).
\end{align*}
Then one readily obtains the desired conclusion \eqref{identity-4}.
\end{proof}

\settocdepth{section}

\subsection{Proof of stabilization}\label{Section-proofcontrol}

We accomplish the proof of Theorem~\ref{Theorem-control}, by adapting the ideas from prior works \cite{LWX-24,CXZZ-25} to the KdV setting.

\begin{proof}[\bf Proof of Theorem~\ref{Theorem-control}]

Let $v$ denote the solution of linearized system \eqref{Control-problem} around $w$, with initial data $v_0=u_0-w_0$. In view of Lemma~\ref{lem-LFtoS(t)v0}, for every $m\in\N^+$, there exists a constant $N=N(m)\in\N^+$ and a control $\xi$ (which automatically satisfies the assertion (2)) such that
\[P_m [v(1)-S_a(1)v_0]=0,\quad \int_{t_1}^{t_2}\|\xi(t)\|^2_{\dot{H}^{\sigma}(s_1,s_2)}dt\le C(\sigma,R)\|v_0\|_{\dot{L}^2}^2.\]
We emphasize that the constant $C(q,\sigma,R)$ do not depend on $m,N$. Meanwhile, note that the operator norm of $\chi \mathcal{P}_N$ from $L^2(0,1;\dot{H}^s)$ into itself has a uniform bound (independent of $N$).\footnote{Recall $\beta_k$ are Laplacian eigenfunctions, and they form an orthogonal basis of $H^s(s_1,s_2)$ for $s\in (-1/2,1/2)$.} Therefore, Corollary~\ref{Coro-HFdissipation} implies that
\begin{align*}
&\|(Id-P_m) [v(1)-S_a(1)v_0]\|_{\dot{L}^2}\le C(R)\langle m\rangle^{-\sigma}\left(
\|v_0\|_{\dot{L}^2}+\|\chi\mathcal{P}_N\xi\|_{L^2(0,1;\dot{H}^{\sigma})} \right)\\
\le &C(R)\langle m\rangle^{-\sigma}\left(
\|v_0\|_{\dot{L}^2}+\|\xi\|_{L^2(t_1,t_2;\dot{H}^{\sigma}(s_1,s_2))}\right)\le C(\sigma,R)\langle m\rangle^{-\sigma}\|v_0\|_{\dot{L}^2}.
\end{align*}
Therefore, we conclude that
\begin{equation}\label{contracting-v}
\begin{aligned}
    &\|v(1)\|_{\dot{L}^2}\leq \|v(1)-S_a(1)v_0\|_{\dot{L}^2}+\|S_a(1) v_0\|_{\dot{L}^2}\\
    \le &\|(Id-P_m)[v(1)-S_a(1)v_0]\|_{\dot{L}^2}+e^{-a} \|v_0\|_{\dot{L}^2}\le \left(C(\sigma,R)m^{-\sigma}+e^{-a}\right)\|v_0\|_{\dot{L}^2}.
\end{aligned}
\end{equation}

In what follows, we choose the same control $\xi$ for the nonlinear controlled system \eqref{Control-nonlinearproblem} (and hence the assertion (2) remains true), which determine the controlled trajectory $u$. We introduce the function $z=u-w-v$, which verifies the equation
\begin{equation*}
\begin{cases}
\partial_t z+\partial^3_x z+a z+\partial_x(wz)=-\frac{1}{2}\partial_x ((u-w)^2),\\
z(0,\cdot)=0.
\end{cases}
\end{equation*}
The local Lipschitz continuity of solution map implies
\begin{equation*}
\|u-w\|_{\dot{Y}_1^0}\le C(R)\|u_0-w_0\|_{\dot{L}^2}.
\end{equation*}
Therefore, by Proposition~\ref{Prop-linear-1} and Remark~\ref{rmk_Zssource}, we get
\[\|z(1)\|_{\dot{L}^2}\le C(R)\|\partial_x ((u-w)^2)\|_{\dot{Z}^0_1}\le C(R)\|u-w\|_{\dot{Y}^0_1}^2\le C(R)\|u_0-w_0\|^2_{\dot{L}^2}.\]
Consequently, recalling \eqref{contracting-v} and $\|u_0-w_0\|_{\dot{L}^2}\le d$, we find
\[\|u(1)-w(1)\|_{L^2}\leq \|z(1)\|_{L^2}+\|v(1)\|_{L^2}\leq \left(e^{-a}+C(\sigma,R)(m^{-\sigma}+d)\right)\|u_0-w_0\|_{L^2}.\]
Therefore, given any $q\in(e^{-a},1)$, one can choose $m\in \N^+$ sufficiently large and $d>0$ sufficiently small, so that $e^{-a}+C(\sigma,R)(m^{-\sigma}+d)\le q$. The proof of Theorem~\ref{Theorem-control} is then complete.
\end{proof}

\section{Exponential mixing of random equation}\label{Section-EM}

We are in a position to accomplish the proof of the Main Theorem, by combining the results from previous sections with a general criterion for exponential mixing. In Section~\ref{Section-criterion}, we recall this criterion relying on coupling methods \cite{KS-12,Shi-15} and asymptotic compactness \cite{LWX-24,CXZZ-25}. Then in Sections~\ref{Section-AC} and~\ref{Section-coupling}, we verify the associated hypotheses for the random KdV model \eqref{Problem-random}.

\subsection{Abstract criterion}\label{Section-criterion}

Let $(\mathcal X,\|\cdot\|)$ and $(\mathcal{Z},\|\cdot\|_{\mathcal{Z}})$ be separable Banach spaces. Let 
\[S\colon\mathcal{X} \times\mathcal{Z} \rightarrow \mathcal{X}\] 
be a locally Lipschitz map, and $\{\eta_n;n\in\N\}$ be a sequence of $\mathcal{Z}$-valued i.i.d.~random variables. The common law of $\eta_n$ is $\ell$, whose support is denoted by $\mathcal{E}$ (which is later assumed to be compact). The Markov process considered here is given by 
\begin{equation}\label{RDS}
x_{n}=S(x_{n-1},\eta_{n-1}),\quad x_0=x\in\mathcal{X}.
\end{equation}
In order to indicate the initial condition and the random inputs, we also write 
\begin{equation*}
x_n=S_{n}(x;\xi_0,\cdots,\xi_{n-1})=S_{n}(x;\boldsymbol{\xi}),\quad \boldsymbol{\xi}:=(\xi_n;n\in\N).
\end{equation*}

With the above setting, system \eqref{RDS} defines a Feller family of Markov chains $(x_n,\Pb_x)$ in $\mathcal{X}$; see, e.g.~\cite[Section~1.3]{KS-12}. We denote the corresponding Markov transition functions by
\begin{equation*}
P_n(x,A)=\Pb_x (x_n\in A),\quad A\in\mathcal{B}(\mathcal{X}),\ n\in \N,
\end{equation*}
the Markov semigroup by
\[P_n\colon C_b(\mathcal{X})\rightarrow C_b(\mathcal{X}),\quad P_nf(x)= \int_{\mathcal{X}} f(y) P_n(x,dy)\] and its dual semigroup by
\[P^*_n\colon \mathcal{P}(\mathcal{X})\rightarrow \mathcal{P}(\mathcal{X}),\quad P_n^*\mu=\int_{\mathcal{X}}P_n(x,\cdot)\mu(dx).\]
We say $\mu\in \mathcal{P}(\dot{L}^2(\T))$ is \textit{invariant} for this Markov process if 
$P_n^*\mu=\mu$ for any $n\in\N^+$.

Below is a list of hypotheses regarding the abstract criterion for exponential mixing. A subset $\mathcal{Y}\subset \mathcal{X}$ is said to be invariant if $S(\mathcal{Y},\mathcal{E})\subset \mathcal{Y}$.

\begin{itemize}
\item [($\mathbf{EAC}$)]   ({\bf Exponential asymptotic compactness}) There exists a compact invariant subset $\mathcal{Y}$ of $\mathcal{X}$, a constant $\kappa>0$, and a measurable function $V\colon\mathcal{X}\rightarrow\R^+$, which sends bounded sets into bounded sets, such that for any $x\in\mathcal{X}$ and $\boldsymbol{\xi}\in \mathcal{E}^{\N}$,
\begin{equation*}
\text{dist}_{\mathcal X}(S_n(x;\boldsymbol{\xi}),\mathcal{Y})\leq V(x)e^{-\kappa n},\quad \forall n\in \N.
\end{equation*}

\item [($\mathbf{I}$)]   ({\bf Irreducibility on $\mathcal{Y}$}) There exists a point $z\in\mathcal{Y}$ such that for every $\varepsilon>0$, there is an integer $m\in\N^+$ satisfying
\begin{equation*}
\inf_{y\in\mathcal{Y}}P_m(y,B_{\mathcal{X}}(z,\varepsilon))>0.
\end{equation*}
		
\item [($\mathbf{C}$)]   ({\bf Coupling condition on $\mathcal Y$}) There exist constants $q\in[0,1)$ and $C>0$, such that for every $y_1,y_2\in\mathcal{Y}$, one can find a coupling $(\mathcal{R}(y_1,y_2),\mathcal{R}'(y_1,y_2))\in\mathscr{C}(P_1(y_1,\cdot),P_1(y_2,\cdot))$ on a common probability space $(\Omega,\mathcal{F},\Pb)$, satisfying
\begin{equation*}
\Pb(\|\mathcal{R}(y_1,y_2)-\mathcal{R}'(y_1,y_2)\|>q\|y_1-y_2\|)\leq C\|y_1-y_2\|.
\end{equation*}
Moreover, the maps $\mathcal{R},\mathcal{R}'\colon\mathcal{Y}\times \mathcal{Y}\times \Omega\rightarrow \mathcal{X}$ are measurable.
\end{itemize}

\noindent We emphasize that in hypothesis $(\mathbf{EAC})$, the trajectory $S_n(x;\boldsymbol{\xi})$ need not enter $\mathcal{Y}$. As realized by \cite{LWX-24}, exponential attraction allows us to reduce the issue to $\mathcal{Y}$, and then the mixing on $\mathcal{Y}$ is more familiar to the literature (see, e.g.~\cite{KS-12}). The following exponential mixing criterion overcomes the lack of smoothing/compactness, and is thus applicable to dispersive PDEs (we refer to \cite{LWX-24} for wave equation, and \cite{CXZZ-25} for NLS equation).

\begin{proposition}[{\cite[Theorem 2.1]{LWX-24}}]\label{Proposition-EX}
Assume that the support $\mathcal{E}$ of $\ell$ is compact, and hypotheses $(\mathbf{EAC})$, $(\mathbf{I})$ and $(\mathbf{C})$ are satisfied. Then the Markov process $(x_n,\Pb_x)$ admits a unique invariant measure $\mu \in \mathcal{P}(\mathcal{X})$, which is supported in $\mathcal{Y}$. Moreover, there exist constants $C,\gamma >0$ such that
\[\|P_n^* \nu-\mu\|_L^* \le Ce^{-\gamma n}\left(1+\int_{\mathcal{X}} V(x)\nu(dx)\right),\quad \forall n\in \N,\]
for any $\nu\in \mathcal{P}(\mathcal{X})$ satisfying $\int_{\mathcal{X}} V(x)\nu(dx) < \infty$.
\end{proposition}

To conclude this subsection, let us fit the randomly forced KdV \eqref{Problem-random} and the noise setting in Section~\ref{Section-mainresult} into the abstract setting of Proposition~\ref{Proposition-EX}:

\begin{itemize}
    \item[\tiny$\bullet$]  Set $\mathcal X=\dot L^2(\T)$ and $\mathcal{Z}$ as the closed linear span of $\chi \alpha_j(t)\beta_k(x)$ in $L^2(0,1;\dot{H}^{1/4+\sigma})$, namely
    \[\mathcal Z=\overline{{\rm span}\{\chi \alpha_j(t) \beta_k(x);j,k\in \N^+\}}^{L^2(0,1;\dot{H}^{1/4+\sigma})}.\]
    According to \eqref{noisestrength}, the noise $\eta_n$ is an $\mathcal{Z}$-valued random variable. Moreover, the support $\mathcal{E}$ of its law $\ell=\mathscr{D}(\eta_n)\in \mathcal{P}(\mathcal{Z})$ is compact in $\mathcal{Z}$ by diagonal argument.

    \item[\tiny$\bullet$] The map $S\colon \mathcal{X}\times \mathcal{Z}\to \mathcal{X}$ is taken as the restriction of the time-$1$ solution map
\[\dot{L}^2\times L^2(0,1;\dot{H}^{1/4+\sigma})\rightarrow \dot{L}^2,\quad S(u_0,f)=u(1),\]
where $u\in \dot{Y}^0_1$ stands for the solution of \eqref{Problem-nonlinear}. In view of Proposition~\ref{Prop-uncontrolled}, $S$ is well-defined, locally Lipschitz and continuously differentiable. 
\end{itemize}

\noindent Therefore, the proof of the Main Theorem reduces to verifying $(\mathbf{EAC})$, $(\mathbf{I})$ and $(\mathbf{C})$.

\subsection{Asymptotic dynamics}\label{Section-AC}

In this subsection, we concentrate on the first two hypotheses $(\mathbf{EAC})$ and $(\mathbf{I})$. The asymptotic compactness is associated with nonlinear smoothing (Proposition~\ref{prop_nonlinearsmoothing}), which indicates that although both $u(t)$ and $S_a(t) u_0$ belong to $\dot{L}^2$, their difference $u(t)-S_a(t)u_0$ is bounded in a higher Sobolev space $H^{1/4+\sigma}$. Since $S_a(t) u_0$ decays exponentially, we find that $u(t)$ is attracted to a bounded subset of $H^{1/4+\sigma}$ at exponential rate, even though $u(t)$ never gains extra regularity.

\begin{proof}[\bf Verification of hypothesis ($\mathbf{EAC}$)]
By the assumption on the strength of noise \eqref{noisestrength},
\begin{equation}\label{mathcalEbound}
\|h\|_{L^2(0,1;\dot{H}^{1/4+\sigma})}\le \mathbf{B},\quad \forall h\in \mathcal{E},
\end{equation}
where $\mathbf{B}$ is a given constant.
For every $\boldsymbol{h}=(h_n;n\in\N )\in \mathcal{E}^{\N}$, let $u\in C(\R^+;\dot{L}^2(\T))$ be the global solution of \eqref{Problem-nonlinear} with initial data $u_0$ and source term $h_n$ on $[n-1,n)$. By iterating the $L^2$ a priori bound \eqref{a priori-estimate-1}, we find
\begin{equation}\label{Absorbing-1}
\|u(t)\|^2_{\dot{L}^2}
\leq e^{-at}\|u_0\|^2_{\dot{L}^2}+C(\mathbf{B}).
\end{equation}
The remainder of proof is decomposed into two steps.

\medskip

\textit{Step 1 $($asymptotic compactness from bounded sets$)$.} 
Fix any $R>0$, let us prove that
\begin{equation}\label{AC-inequality}
    \|u(t)-S_a(t) u_0\|_{\dot{H}^{1/4+\sigma}}\le C(\mathbf{B},R),\quad \forall u_0\in B_{\dot{L}^2(R)},\ n\in \N.
\end{equation}
And consequently, due to \eqref{S_a(t)decay}, we find
\begin{equation}\label{AC-reduction}
    \dist_{\dot{L}^2} (u(t),\mathcal{K}_R)\le Re^{-at},\quad \text{where }\mathcal{K}_R:=B_{\dot{H}^{1/4+\sigma}}(C(\mathbf{B},R)).
\end{equation}
Note that $\mathcal{K}_R$ is compact in $\dot{L}^2$, since $\dot{H}^{1/4+\sigma}$ is compactly embedded in $\dot{L}^2$. This is slightly weaker than $(\mathbf{EAC})$, since $\mathcal{K}_R$ depends on $R$ and is not necessarily invariant.

In fact, \eqref{Absorbing-1} implies $\|u(t)\|_{\dot{L}^2}\leq C(\mathbf{B},R)$ for any $t\geq 0$. Then Proposition~\ref{prop_nonlinearsmoothing} and \eqref{mathcalEbound} yield
\begin{equation}\label{ACinequality-1}
\|u(n+t)-S_a(t)u(n)\|_{\dot{H}^{1/4+\sigma}}\le C(\mathbf{B},R),\quad \forall t\in [0,1],\ n\in \N.
\end{equation}
Therefore, for any $n\in \N$, it follows that
\begin{equation}\label{ACinequality-2}
\begin{aligned}
&\|u(n)-S_a(n)u_0\|_{\dot{H}^{1/4+\sigma}}=\|\sum_{j=1}^n \left[
S_a(n-j)u(j)-S_a(n-j+1)u(j-1)
\right]\|_{\dot{H}^{1/4+\sigma}}\\
\leq &\sum_{j=1}^n e^{-a(n-j)} \|u(j)-S_a(1)u(j-1)\|_{\dot{H}^{1/4+\sigma}}\le  C(\mathbf{B},R)\sum_{j=0}^\infty e^{-a j}\leq C(\mathbf{B},R).
\end{aligned}
\end{equation}
In general, let us write $t>0$ by $t=n+t'$, where $n\in \N^+$ and $t'\in[0,1)$. Then
\begin{equation*}
\|u(t)-S_a(t)u_0\|_{\dot{H}^{1/4+\sigma}}
\leq \|u(n+t')-S_a(t')u(n)\|_{\dot{H}^{1/4+\sigma}} +\|u(n)-S_a(n)u_0\|_{\dot{H}^{1/4+\sigma}}
\end{equation*}
Here, the first term on the right-hand side is bounded by \eqref{ACinequality-1}, and the second by \eqref{ACinequality-2}. This leads to \eqref{AC-inequality}, and hence \eqref{AC-reduction}.

\medskip
	
\textit{Step 2 $($Absorbing to a bounded set$)$.} Owing to \eqref{Absorbing-1}, one can readily find constants $C_0>0$ and $R=R(\mathbf{B})>0$, such that for any $u_0\in \dot{L}^2$,
\[\|u(t)\|_{\dot{L}^2}<R,\quad \forall t\ge C_0 \log (1+\|u_0\|_{\dot{L}^2}).\]
Let us fix $\mathcal{K}=\mathcal{K}_R$, which satisfies \eqref{AC-reduction}, and a small constant $\gamma\in (0,a)$ (to be chosen later). Then for any $u_0\in\dot{L}^2$ and $t\geq T:=C_0\log \left(1+\|u_0\|_{\dot{L}^2}\right)$, 
\[\dist_{L^2}(u(t),\mathcal K)\le Ce^{-a(t-T)}\le Ce^{-\gamma(t-T)}=C\left(1+\|u_0\|_{L^2}\right)^{C_0\gamma}e^{-\gamma t}.\]
Meanwhile, for $t\leq T$, we invoke \eqref{Absorbing-1} again to check that
\[\dist_{\dot{L}^2}(u(t),\mathcal K)\le C+\|u(t)\|_{L^2}\le C\left(1+\|u_0\|_{L^2}\right)e^{\gamma T}e^{-\gamma t}=C\left(1+\|u_0\|_{L^2}\right)^{1+C_0\gamma}e^{-\gamma t}.\]
In conclusion, 
\[\dist_{L^2}(u(t),\mathcal K)\le C\left(1+\|u_0\|_{L^2}\right)^{1+C_0\gamma}e^{-\gamma t},\ \forall t\geq 0.\]
This gives rise to the desired property ($\mathbf{EAC}$) with $V=C(1+\|u_0\|^2_{L^2})$,
by taking $\gamma$ small enough so that $C_0\gamma\leq 1$, and 
$\mathcal Y\supset \mathcal{K}$ to be the attainable set from $\mathcal K$:
\[
\mathcal Y=\overline{\bigcup_{n\in\N}\mathcal{A}_n}\quad \text{with }\mathcal{A}_0=\mathcal K,\ \mathcal A_n=\{S(v,h);v\in \mathcal{A}_{n-1},h\in \mathcal{E}\}\,(n\geq 1).\]
It is not hard to see that $\mathcal{Y}$ is compact and invariant; see, e.g.~\cite[Proposition 2.2]{LWX-24}.
\end{proof}

\begin{proof}[\bf Verification of hypothesis ($\mathbf{I}$)]
If we take $f(t,x)\equiv 0$ in \eqref{a priori-estimate-1}, then it follows that for every $\varepsilon>0$, there exists an integer $m\in\N^+$ such that
\begin{equation*}
\|S_m(u_0;\boldsymbol{0})\|_{\dot{L}^2}<\frac{\varepsilon}{2}
\end{equation*} 
for any $u_0\in\mathcal Y$, where $\boldsymbol{0}$ stands for a sequence of zeros.
Since the map 
\[\mathcal Y\times\mathcal E^m\ni(u_0,\boldsymbol{h})\mapsto S_m(u_0;\boldsymbol{h})\in \dot{L}^2(\T)\]
is Lipschitz continuous (Proposition~\ref{Prop-uncontrolled}), there exists a constant $\delta>0$ such that 
\begin{equation*}
\|S_m(u_0;\boldsymbol{h})\|_{\dot{L}^2}<\varepsilon
\end{equation*}
for any $u_0\in\mathcal Y$ and $\boldsymbol{h}\in \mathcal{E}^\N$ with $h_0,\cdots, h_{m-1}\in B_{L^2(0,1;\dot{L}^2)}(\delta)$. We then conclude that 
\begin{equation*}
\begin{aligned}
P_m(u_0,B_{\dot{L}^2}(\varepsilon)) \geq \Pb (\|\eta_n\|_{L^2(0,1;\dot{L}^2)}<\delta,\,\forall\,0\leq n\leq m-1) = \ell (B_{L^2(0,1;\dot{L}^2)}(\delta))^{m}>0.
\end{aligned}
\end{equation*}
The last step is due to $0\in \mathcal E$, which is assured by $\rho_{j,k}(0)>0$. 
\end{proof}

\subsection{Coupling property via control}\label{Section-coupling}

We proceed to verify the coupling property for the Markov process $u_n$,  making use of the control property established in Theorem~\ref{Theorem-control}.

\begin{proof}[\bf Verification of hypothesis ($\mathbf{C}$)]
Let us denote the finite dimensional subspace
\begin{equation*}
\mathcal{Z}_1={\rm span}\{\chi \alpha_j(t) \beta_k(x);1\leq j,k\le N\}.
\end{equation*}
In view of \cite[Lemma 5.3]{CXZZ-25}, the conclusion of this proposition follows from the following property: 
\begin{itemize}
\item[\tiny$\bullet$] There exist constants $d>0$ and $q\in (0,1)$ such that for any $(u_0,u_0')\in \mathcal{Y}\times \mathcal{Y}$ with $\|u_0-u_0'\|_{L^2}\le d$, there exists a continuously differentiable map
$\Phi^{u_0,u_0'}\colon \mathcal{Z}\to \mathcal{Z}_1$, such that for $\ell$-a.e.~$h\in \mathcal{Z}$, we have
\begin{equation}\label{coupling-condition-1}
    \|S(u_0,h)-S(u_0',h+\Phi^{u_0,u_0'}(h))\|_{\dot{L}^2}\le q \|u_0-u_0'\|_{\dot{L}^2},
\end{equation}
and in addition,
\begin{equation}\label{coupling-condition-2}
    \sup_{h\in \mathcal{Z}}\|\Phi^{u_0,u_0'}(h)\|_{\dot{L}^2(0,1; \dot{H}^{1/4+\sigma})} \lesssim\|u_0-u_0'\|_{\dot{L}^2},\quad  \Lip(\Phi^{u_0,u_0'})\lesssim\|u_0-u_0'\|_{L^2}. 
\end{equation}
\end{itemize}
If this is true, then $(\mathbf{C})$ holds for any slightly larger $q$ (still strictly less than $1$). Accordingly, the remaining task is to verify \eqref{coupling-condition-1} and \eqref{coupling-condition-2}.

Choose $R>2$ sufficiently large so that 
\begin{equation*}
\mathcal Y\subset B_{\dot{H}^{1/4+\sigma}}(R-2)\quad \text{and}\quad \mathcal E\subset B_{L^2(0,1;\dot{H}_x^{1/4+\sigma})}(R-2).
\end{equation*}
We then apply Theorem~\ref{Theorem-control} with fixed $q\in (e^{-a},1)$ and such $R$, to deduce that there exist constants $N\in\N^+$, $d>0$, and a map 
\[\Phi'\colon B_{\dot{H}^{1/4+\sigma}}(R)\times B_{L^2(0,1;\dot{H}^{1/4+\sigma})}(R)\rightarrow \mathcal{L}(\dot{L}^2(\T);L^2(t_1,t_2;\dot{L}^2(s_1,s_2)))\] 
such that 
\begin{equation}\label{Control-squeezing}
\|S(u_0,h)-S(u_0',h+\chi\mathcal{P}_N\Phi'(u_0,h)(u_0-u_0'))\|_{L^2}\leq q\|u_0-u_0'\|_{L^2}
\end{equation}
for any $u_0,u_0'\in B_{\dot{H}^{1/4+\sigma}(R)}$ with $\|u_0-u_0'\|_{L^2}\leq d$ and $h\in B_{L^2(0,1;\dot{H}^{1/4+\sigma})}(R)$. Moreover, the map $\Phi'$ is Lipschitz and continuously differentiable in its domain. 

Now, the desired map $\Phi^{u_0,u_0'}$ can be constructed as
\[\Phi^{u_0,u_0'}(h)=\phi\left(\|h\|^2_{L^2(0,1;\dot{H}^{1/4+\sigma})}\right)\chi\mathcal{P}_N \Phi'(u_0,h)(u_0-u_0'),\]
where $\phi\in C_0^\infty(\R)$ with $\phi(r)=1$ for $|r|\le (R-2)^2$ and $\phi(r)=0$ for $|r|>(R-1)^2$. Then \eqref{coupling-condition-1} is derived from \eqref{Control-squeezing}, and the $C^1$ regularity of $\Phi^{u,u'}$ and \eqref{coupling-condition-2} are not hard to justify thanks to the corresponding estimates for $\Phi'$.
\end{proof}

Having verified all the hypotheses of Proposition~\ref{Proposition-EX}, the Main Theorem follows immediately.


\appendix


\section{Supplementary ingredients for observability inequality}\label{Appendix-2}

This appendix collects some auxiliary results used in Section~\ref{Section-control}.

\subsection{Profile of linear equations}\label{Appendix-linearequation} 

As shown in Section~\ref{Section-control}, the control for stabilizing the nonlinear system will be obtained by analyzing the linearization \eqref{Control-problem} and its dual version \eqref{Adjoint-problem}. We present basic characterizations of these equations.

Let us introduce the linearized equation of the form
\begin{equation}\label{Problem-linear-1}
\begin{cases}
    \partial_t v+\partial^3_xv+av+\partial_x(wv)=f,\quad t\in[0,T],\\
    v(0,\cdot)=v_0,
\end{cases}
\end{equation}
which coincides with \eqref{Control-problem} by taking $f=\chi\mathcal{P}_n\xi$. 
Another problem to be considered reads 
\begin{equation}\label{Problem-linear-2}
\begin{cases}
    \partial_t \varphi+\partial^3_x\varphi-a\varphi+[w\partial_x\varphi-\widehat{w\partial_x\varphi}(0)]=0,\quad t\in[0,T],\\
    \varphi(T,\cdot)=\varphi_T,
\end{cases}
\end{equation}
which coincides with the dual system \eqref{Adjoint-problem}.

\begin{proposition}\label{Prop-linear-1}
Given $T,R>0$ and $s\in [0,1)$, there exists a constant $C>0$, such that for any $v_0\in\dot{H}^s$, $f\in L^2(0,T;\dot{H}^s)$ and $w\in B_{\dot{Y}^s_T}(R)$, problem \eqref{Problem-linear-1} admits a unique solution $v=:\mathcal{V}_w(v_0,f)\in \dot{Y}^s_T$, satisfying
\begin{equation}\label{Vw(v0,f)est}
    \|v\|_{\dot{Y}^s_T}\leq C\left(\|v_0\|_{\dot{H}^s}+\|f\|_{L^2(0,T;\dot{H}^s)}\right).
\end{equation}
Moreover, the map 
\[B_{\dot{Y}^s_T}(R)\times B_{L^2(0,T;\dot{H}^s)}(R)\ni (w,f)\mapsto \mathcal V_w(\cdot,f)\in\mathcal L(\dot{H}^s; \dot{Y}^s_T)\]
is Lipschitz and continuously differentiable.
\end{proposition}

\begin{remark}\label{rmk_Zssource}
This result can be slightly generalized to $f\in \dot{Z}^s_T$, while the proof remains valid. In particular, the term $\|f\|_{L^2(0,T;\dot{H}^s)}$ in \eqref{Vw(v0,f)est} can be replaced by $\|f\|_{\dot{Z}^s_T}$.
\end{remark}

\begin{proposition}\label{Prop-linear-2}
Given $T,R,s^*>0$ and $s\in (-s^*,s^*)$, there exists a constant $C>0$, such that for any $\varphi_T\in\dot{H}^{-s}$ and $w\in B_{\dot{Y}^{s^*}_T}(R)$, problem \eqref{Problem-linear-2} admits a unique solution $\varphi=:\mathcal U_w(\varphi_T)\in \dot{Y}^{-s}_T$, satisfying
\[\|\varphi\|_{\dot{Y}^{-s}_T}\leq C\|\varphi_T\|_{\dot{H}^{-s}}.\]
Moreover, the map 
\[B_{\dot{Y}^{s^*}_T}(R)\ni w\mapsto \mathcal U_w(\cdot)\in\mathcal L(\dot{H}^{-s};\dot{Y}^{-s}_T)\]
is Lipschitz and continuously differentiable.
\end{proposition}

The proof of Proposition~\ref{Prop-linear-1} is much simpler than that for the nonlinear problem \eqref{Problem-nonlinear} (see Proposition~\ref{Prop-uncontrolled}). Meanwhile, the proof of Proposition~\ref{Prop-linear-2} is based on $Z^s_T$ estimate for the bilinear operator $w\partial_x\varphi-\widehat{w\partial_x\varphi}(0)$, which is collected in Lemma~\ref{Lemma-bilinear} below. With this lemma in hand, it is not hard to apply the fixed point argument to conclude the proof of Proposition~\ref{Prop-linear-2}; the details are omitted.

\begin{lemma}\label{Lemma-bilinear}
Given $s^*>0$ and $s\in (-s^*,s^*)$, there exists a constant $C>0$, such that for any $T>0$,
\begin{align*}
\|w\partial_x\varphi-\widehat{w\partial_x\varphi}(0)\|_{\dot{Z}^s_T}&\le C (T^{1/6-}\wedge 1) \|w\|_{\dot{X}^{s^*,1/2}_T}\|\varphi\|_{\dot{X}^{s,1/2}_T}.
\end{align*}
\end{lemma}

This lemma is obtained by adapting the arguments for the classical  estimate in Lemma~\ref{lem_Bougainproperty}(3) (cf.~\cite{CKSTT-03,ET-16}). Nevertheless, unlike the case of Schr\"odinger equations in \cite{CXZZ-25}, one cannot directly derive this estimate from the classical by taking duals. For the sake of completeness, we provide the details of proof in the sequel. We mention that a natural question is if this result remains valid for $s=\pm s^*$, while the weaker version stated here is sufficient for the current paper.

\begin{proof}[\bf Proof of Lemma~\ref{Lemma-bilinear}]
It suffices to check that, for any $w\in \dot{X}^{s^*,1/2}$ and $\varphi\in \dot{X}^{s,1/2}$ (without time restriction), we have
\begin{equation}\label{eq-bilinear-1}
\|w\partial_x\varphi-\widehat{w\partial_x\varphi}(0)\|_{\dot{Z}^s}\le C \left[\|w\|_{\dot{X}^{s^*,1/2}} \|\varphi\|_{\dot{X}^{s,1/3}}+\|w\|_{\dot{X}^{s^*,1/3}} \|\varphi\|_{\dot{X}^{s,1/2}}\right].
\end{equation}
Indeed, this implies the corresponding estimate with the same constant $C$ on $[0,T]$. Then according to Lemma~\ref{Lemma-Bourgain}(4),
the conclusion of this lemma follows immediately.

To this end, we only deal with $s\in (-s^*,0]$, while the case of $s\in [0,s^*)$ is analogous (and simpler). Set $\rho:=-s\in [0,s^*)$. We need to estimate the two parts of the $\dot{Z}^s$ norm \eqref{def_Z^s}.

\medskip

{\it Estimate for the first part of $Z^s$ norm:} Since $w\partial_x\varphi-\widehat{w\partial_x\varphi}(0)$ has mean zero in space, by duality and $\rho=-s$, we can represent $\|w\partial_x\varphi-\widehat{w\partial_x\varphi}(0)\|_{\dot{X}^{s,-1/2}}$ as
\begin{equation}\label{bound-1}
\sup_{\|u\|_{\dot{X}^{\rho,1/2}}=1} \left|\sum_{k_1+k_2+k_3=0}\int_{\tau_1+\tau_2+\tau_3=0} k_2 \widehat{w}(\tau_1,k_1)\widehat{\varphi}(\tau_2,k_2)\widehat{u}(\tau_3,k_3)\right|.
\end{equation}
Note that we always have $k_1,k_2,k_3\not =0$ in the summation.
Let us define the auxiliary functions
\begin{equation*}
\begin{aligned}
&f_1(\tau,k)=\langle k\rangle^{\rho} \langle\tau-k^3 \rangle^{1/2}|\widehat{w}(\tau,k)|,\\
&f_2(\tau,k)=\langle k\rangle^{-\rho} \langle\tau-k^3 \rangle^{1/2}|\widehat{\varphi}(\tau,k)|,\\
&f_3(\tau,k)=\langle k\rangle^{\rho} \langle\tau-k^3 \rangle^{1/2}|\widehat{u}(\tau,k)|.
\end{aligned}
\end{equation*}
Then \eqref{bound-1} (and thus $\|w\partial_x\varphi-\widehat{w\partial_x\varphi}(0)\|_{\dot{X}^{s,-1/2}}$) can be bounded by
\begin{equation}\label{bound-2}
\sum_{k_1+k_2+k_3=0}\int_{\tau_1+\tau_2+\tau_3=0}Mf_1(\tau_1,k_1)f_2(\tau_2,k_2)f_3(\tau_3,k_3)d\tau_1d\tau_2d\tau_3 ,
\end{equation}
where the multiplier $M=M(k_1,k_2,k_3,\tau_1,\tau_2,\tau_3)$ takes the form
\[M:=\frac{\langle k_2\rangle^{1+\rho}}{\langle k_1\rangle^\rho\langle k_3\rangle^{\rho}\langle\tau_1-k_1^3\rangle^{1/2}\langle\tau_2-k_2^3\rangle^{1/2}\langle\tau_3-k_3^3\rangle^{1/2}}.\]
Due to $k_j\neq 0$ and 
$
\sum_{j=1}^{3}(\tau_j-k_j^3)=-3k_1k_2k_3,
$
it follows that
\begin{equation}\label{maxtau-k^3}
    \max \{\langle \tau_j-k_j^3\rangle\}\gtrsim \langle k_1\rangle \langle k_2\rangle \langle k_3\rangle.
\end{equation}

Let $I_i\ (i=1,2,3)$ stand for the contribution in \eqref{bound-2} of the summand satisfying 
\[\max\{\langle\tau_j-k_j^3\rangle\}=\langle\tau_i-k_i^3\rangle.\]
For the terms in $I_1$, using \eqref{maxtau-k^3} and $\langle k_2\rangle=\langle k_1+k_3\rangle\lesssim \langle k_1\rangle \langle k_3\rangle$, we find
\[M\lesssim \frac{\langle k_2\rangle^{1/2+\rho}}{\langle k_1\rangle^{1/2+\rho}\langle k_3\rangle^{1/2+\rho}\langle\tau_2-k_2^3 \rangle^{1/2}\langle\tau_3-k_3^3 \rangle^{1/2}}\lesssim \frac{1 }{\langle\tau_2-k_2^3 \rangle^{1/2}\langle\tau_3-k_3^3 \rangle^{1/2}}.\]
This together with Cauchy--Schwarz inequality and Lemma~\ref{Lemma-Bourgain}(3) implies that 
\begin{align*}
I_1&\lesssim \sum_{k_1+k_2+k_3=0}\int_{\tau_1+\tau_2+\tau_3=0}\frac{f_1(\tau_1,k_1)f_2(\tau_2,k_2)f_3(\tau_3,k_3)}{\langle\tau_2-k_2^3 \rangle^{1/2}\langle\tau_3-k_3^3 \rangle^{1/2}} d\tau_1 d\tau_2 d\tau_3\\
&=\int_{\R\times \T}\mathcal{F}^{-1}(f_1)\mathcal{F}^{-1}\left(\frac{f_2}{\langle\tau-k^3 \rangle^{1/2}}\right)\mathcal{F}^{-1}\left(\frac{f_3}{\langle\tau-k^3 \rangle^{1/2}}\right)dt dx\\
&\lesssim \|\mathcal{F}^{-1}(f_1)\|_{L^2_{t,x}}\left\|\mathcal{F}^{-1}\left(\frac{f_2}{\langle\tau-k^3 \rangle^{1/2}}\right)\right\|_{L^4_{t,x}}\left\|\mathcal{F}^{-1}\left(\frac{f_3}{\langle\tau-k^3 \rangle^{1/2}}\right)\right\|_{L^4_{t,x}}\\
&\lesssim \|\mathcal{F}^{-1}(f_1)\|_{L^2_{t,x}}\left\|\mathcal{F}^{-1}\left(\frac{f_2}{\langle\tau-k^3 \rangle^{1/2}}\right)\right\|_{X^{0,1/3}}\left\|\mathcal{F}^{-1}\left(\frac{f_3}{\langle\tau-k^3 \rangle^{1/2}}\right)\right\|_{X^{0,1/3}}\\
&\lesssim\|w\|_{\dot{X}^{\rho,1/2}}\|\varphi\|_{\dot{X}^{-\rho,1/3}}\|u\|_{\dot{X}^{\rho,1/3}}\lesssim \|w\|_{\dot{X}^{s^*,1/2}}\|\varphi\|_{\dot{X}^{s,1/3}}.
\end{align*}
We can address $I_2$ and $I_3$ in analogous ways, which lead to
\begin{align*}
I_2&\lesssim \left\|\mathcal{F}^{-1}\left(\frac{f_1}{\langle\tau-k^3 \rangle^{1/2}}\right)\right\|_{X^{0,1/3}} \|\mathcal{F}^{-1}(f_2)\|_{L^2_{t,x}} \left\|\mathcal{F}^{-1}\left(\frac{f_3}{\langle\tau-k^3 \rangle^{1/2}}\right)\right\|_{X^{0,1/3}}\\
&\lesssim \|w\|_{\dot{X}^{\rho,1/3}} \|\varphi\|_{\dot{X}^{-\rho,1/2}}\|u\|_{\dot{X}^{\rho,1/3}}\lesssim \|w\|_{\dot{X}^{s^*,1/3}}\|\varphi\|_{\dot{X}^{s,1/2}},
\end{align*}
and
\begin{align*}
I_3&\lesssim \left\|\mathcal{F}^{-1}\left(\frac{f_1}{\langle\tau-k^3 \rangle^{1/2}}\right)\right\|_{X^{0,1/3}} \left\|\mathcal{F}^{-1}\left(\frac{f_2}{\langle\tau-k^3 \rangle^{1/2}}\right)\right\|_{X^{0,1/3}} \|\mathcal{F}^{-1}(f_3)\|_{L^2_{t,x}}\\
&\lesssim \|w\|_{\dot{X}^{\rho,1/3}} \|\varphi\|_{\dot{X}^{-\rho,1/3}} \|u\|_{\dot{X}^{\rho,1/2}}\lesssim \|w\|_{\dot{X}^{s^*,1/3}} \|\varphi\|_{\dot{X}^{s,1/3}}.
\end{align*}

\medskip

{\it Estimate for the second part of $Z^s$ norm:} Let us rewrite the second part as
\begin{equation}\label{summand-1}
\sup_{\widehat{u}(\tau,0)\equiv 0,\ \|\widehat u\|_{l_k^2L_\tau^\infty}=1} \left|\sum_{k_1+k_2+k_3=0}\int_{\tau_1+\tau_2+\tau_3=0}\frac{k_2 \widehat{w}(\tau_1,k_1)\widehat{\varphi}(\tau_2,k_2)\widehat{u}(\tau_3,k_3)}{\langle k_3\rangle^{\rho}\langle\tau_3-k_3^3\rangle}\right|.
\end{equation}
We still have $k_j\not =0$ in the summation. This can be further bounded by
\begin{equation}\label{summand-2}
    \sup_{\widehat{u}(\tau,0)\equiv 0,\ \|\widehat u\|_{l_k^2L_\tau^\infty}=1}\sum_{k_1+k_2+k_3=0}\int_{\tau_1+\tau_2+\tau_3=0}\widetilde Mf_1(\tau_1,k_1)f_2(\tau_2,k_2)|\widehat{u}(\tau_3,k_3)|,
\end{equation}
where the new multiplier $\widetilde{M}$ is
\[\widetilde M:=\frac{\langle k_2\rangle^{1+\rho}}{\langle k_1\rangle^\rho\langle k_3\rangle^{\rho}\langle\tau_1-k_1^3\rangle^{1/2}\langle\tau_2-k_2^3\rangle^{1/2}\langle\tau_3-k_3^3\rangle} (=M\langle \tau_3-k_3^3\rangle^{-1/2}).\]

Let $\widetilde{I}_i\ (i=1,2,3)$ stand for the contribution in \eqref{summand-2} of the summand satisfying $\max\{\langle\tau_j-k_j^3\rangle\}=\langle\tau_i-k_i^3\rangle$. For the terms in $\widetilde{I}_1$, we use the same arguments as for $M$ to find
\[\widetilde{M}\lesssim \frac{1}{\langle\tau_2-k_2^3\rangle^{1/2}\langle\tau_3-k_3^3\rangle}.\]
One can readily verify $\|u\|_{X^{0,-b}}\lesssim \|\widehat{u}\|_{l^2_k L^\infty_\tau}$ for $b>1/2$, due to $\|\langle \tau\rangle^{-b}\|_{L^2_\tau}<\infty$. Hence
\begin{align*}
\widetilde{I}_1 &\lesssim \int_{\R\times\T}\mathcal{F}^{-1}(f_1)\mathcal{F}^{-1}\left(\frac{ f_2}{\langle\tau-k^3\rangle^{1/2}}\right)\mathcal{F}^{-1}\left(\frac{|\widehat{u}(\tau,k)|}{\langle\tau-k^3\rangle}\right)dtdx\\
& \lesssim\|\mathcal{F}^{-1}(f_1)\|_{L^2_{t,x}}\left\|\mathcal{F}^{-1}\left(\frac{ f_2}{\langle\tau-k^3\rangle^{1/2}}\right)\right\|_{X^{0,1/3}}\left\|\mathcal{F}^{-1}\left(\frac{|\widehat{u}(\tau,k)|}{\langle\tau-k^3\rangle}\right)\right\|_{X^{0,1/3}}\\
& \lesssim \|w\|_{\dot{X}^{\rho,1/2}}\|\varphi\|_{\dot{X}^{-\rho,1/3}}\|u\|_{\dot{X}^{0,-2/3}}\lesssim\|w\|_{\dot{X}^{s^*,1/2}}\|\varphi\|_{\dot{X}^{s,1/3}}.
\end{align*}
And in the same manner, we have
\begin{align*}
    \widetilde{I}_2& \lesssim \left\|\mathcal{F}^{-1}\left(\frac{ f_1}{\langle\tau-k^3\rangle^{1/2}}\right)\right\|_{X^{0,1/3}} \|\mathcal{F}^{-1}(f_2)\|_{L^2_{t,x}} \left\|\mathcal{F}^{-1}\left(\frac{|\widehat{u}(\tau,k)|}{\langle\tau-k^3\rangle}\right)\right\|_{X^{0,1/3}}\\
    &\lesssim \|w\|_{\dot{X}^{\rho,1/3}}\|\varphi\|_{\dot{X}^{-\rho,1/2}}\|u\|_{\dot{X}^{0,-2/3}}\lesssim\|w\|_{\dot{X}^{s^*,1/3}}\|\varphi\|_{\dot{X}^{s,1/2}}.
\end{align*}
The estimate for $\widetilde{I}_3$ is more subtle and relies on the assumption that $\rho=-s<s^*$. Indeed, taking the $[1-(s^*-\rho)]/2$-th power of \eqref{maxtau-k^3}, we find that for any term belonging to $\widetilde{I}_3$,
\begin{align*}
\widetilde{M}&\lesssim \frac{\langle k_2\rangle^{(1+s^*+\rho)/2}}{\langle k_1\rangle^{(1-s^*+3\rho)/2} \langle k_3\rangle^{(1-s^*+3\rho)/2} \langle \tau_1-k_1^3\rangle^{1/2} \langle \tau_2-k_2^3\rangle^{1/2} \langle \tau_3-k_3^3\rangle^{(1+s^*-\rho)/2}}\\
&\lesssim \frac{\langle k_1\rangle^{s^*-\rho}\langle k_3\rangle^{s^*-\rho}}{\langle \tau_1-k_1^3\rangle^{1/2} \langle \tau_2-k_2^3\rangle^{1/2} \langle \tau_3-k_3^3\rangle^{(1+s^*-\rho)/2}}.
\end{align*}
Also note that $\langle k_3\rangle \lesssim \langle k_1\rangle \langle k_2\rangle$ and \eqref{maxtau-k^3} imply
$\langle \tau_3-k_3^3\rangle\gtrsim \langle k_3\rangle^2$.
Therefore
\begin{align*}
\widetilde{I}_3&\lesssim \left\|\mathcal{F}^{-1}\left(\frac{\langle k\rangle^{s^*-\rho} f_1}{\langle\tau-k^3\rangle^{1/2}}\right)\right\|_{X^{0,1/3}} \left\|\mathcal{F}^{-1}\left(\frac{f_2}{\langle\tau-k^3\rangle^{1/2}}\right)\right\|_{X^{0,1/3}}\cdot \left\|\frac{\langle k\rangle^{s^*-\rho} |\widehat{u}(\tau,k)|}{(\langle\tau-k^3\rangle+\langle k\rangle^2)^{(1+s^*-\rho)/2}}\right\|_{l^2_k L^2_\tau}\\
&\lesssim \|w\|_{\dot{X}^{s^*,1/3}} \|\varphi\|_{\dot{X}^{\rho,1/3}} \|\widehat{u}(\tau,k)\|_{l^2_k L^\infty_\tau} \left\|\frac{\langle k\rangle^{s^*-\rho}}{(\langle\tau-k^3\rangle+\langle k\rangle^2)^{(1+s^*-\rho)/2}}\right\|_{l^\infty_k L^2_\tau}.
\end{align*}
Finally, we use the standard estimate
\[\int_{\R} \frac{d\tau}{(\langle \tau\rangle+\langle k\rangle^2)^{1+s^*-\rho}}\lesssim \frac{1}{\langle k\rangle^{2(s^*-\rho)}}.\]
This yields the last $l^\infty_k L^2_\tau$ norm in the estimate for $\widetilde{I}_3$ finite.

In conclusion, we arrive at \eqref{eq-bilinear-1} and the proof is complete. 
\end{proof}

From the above analysis, we can further derive the following modified estimates, which trades some spatial regularity for temporal regularity, and is useful to the proof of full observability in Lemma~\ref{Lemma-fullobs}. The ideas are exactly the same, but some parameters must be adjusted. For the sake of rigor, we provide the details.
	
\begin{corollary}\label{Coro-bilinear}
Given $T,s^*>0$ and $s\in [0,s^*)$, for any
\[\varepsilon\le \min \{(s^*-s)/2,1/6\},\]
there exists a constant $C>0$, such that
\begin{equation}\label{bilinear-modify-1}
\|w\partial_x \varphi-\widehat{w\partial_x \varphi}(0)\|_{\dot{X}^{-s-2\varepsilon,-1/2+\varepsilon}_T}\leq C\|w\|_{\dot{X}^{s^*,1/2}_T}\|\varphi\|_{\dot{X}^{-s,1/2}_T}
\end{equation}
and
\begin{equation}\label{bilinear-modify-2}
\|w\partial_x \varphi-\widehat{w\partial_x \varphi}(0)\|_{\dot{X}^{-s-2\varepsilon,-1/2}_T}\leq C\|w\|_{\dot{X}_T^{s^*,1/2-\varepsilon}}\|\varphi\|_{\dot{X}^{-s,1/2-\varepsilon}_T}.
\end{equation}
\end{corollary}
	
\begin{proof}[{\bf Proof of Corollary~\ref{Coro-bilinear}}]
To establish \eqref{bilinear-modify-1}, by dual method, we have
\[\text{LHS of }\eqref{bilinear-modify-1}\leq \sup_{\|u\|_{\dot{X}^{s+2\varepsilon,1/2-\varepsilon}}=1} \sum_{k_1+k_2+k_3=0}\int_{\tau_1+\tau_2+\tau_3=0}Mf_1(\tau_1,k_1)f_2(\tau_2,k_2)f_3(\tau_3,k_3).\]
Here $k_1,k_2,k_3\not =0$ in the summation,
\begin{equation*}
\begin{aligned}
&f_1(\tau,k)=\langle k\rangle^{s^*} \langle\tau-k^3 \rangle^{1/2}|\widehat{w}(\tau,k)|,\\
&f_2(\tau,k)=\langle k\rangle^{-s} \langle\tau-k^3 \rangle^{1/2}|\widehat{\varphi}(\tau,k)|,\\
&f_3(\tau,k)=\langle k\rangle^{s+2\varepsilon} \langle\tau-k^3 \rangle^{1/2-\varepsilon}|\widehat{u}(\tau,k)|,
\end{aligned}
\end{equation*}
and
\[M=\frac{\langle k_2\rangle^{1+s}}{\langle k_1\rangle^{s^*} \langle k_3\rangle^{s+2\varepsilon}\langle\tau_1-k_1^3 \rangle^{1/2}\langle\tau_2-k_2^3 \rangle^{1/2}\langle\tau_3-k_3^3 \rangle^{1/2-\varepsilon}}.\]
Let $I_i$ stand for the contribution in the case of $\max\{\langle\tau_j-k_j^3\rangle\}=\langle\tau_i-k_i^3\rangle$.
The estimates for $I_1,I_2$ are almost the same as in Lemma~\ref{Lemma-bilinear}. For $I_3$, provided $\varepsilon\le (s^*-s)/2$, we have
\begin{align*}
M&\lesssim \frac{\langle k_2\rangle^{1/2+s+\varepsilon}}{\langle k_1\rangle^{1/2+s^*-\varepsilon} \langle k_3\rangle^{1/2+s+\varepsilon} \langle\tau_1-k_1^3 \rangle^{1/2}\langle\tau_2-k_2^3 \rangle^{1/2}}\lesssim \frac{1}{\langle\tau_1-k_1^3 \rangle^{1/2}\langle\tau_2-k_2^3 \rangle^{1/2}},
\end{align*}
which leads to
\[I_3\lesssim \|w\|_{\dot{X}^{s^*,1/3}}\|\varphi\|_{\dot{X}^{-s,1/3}}\|u\|_{\dot{X}^{s+2\varepsilon,1/2-\varepsilon}}\lesssim \|w\|_{\dot{X}^{s^*,1/3}}\|\varphi\|_{\dot{X}^{-s,1/3}}.\]
This finishes the proof of \eqref{bilinear-modify-1}. 

Next we turn to \eqref{bilinear-modify-2}. By dual method,
\[\text{LHS of }\eqref{bilinear-modify-2}=\sup_{\|u\|_{\dot{X}^{s,1/2}}=1} \sum_{k_1+k_2+k_3=0}\int_{\tau_1+\tau_2+\tau_3=0} \widetilde{M} \widetilde{f}_1(\tau_1,k_1) \widetilde{f}_2(\tau_2,k_2) \widetilde{f}_3(\tau_3,k_3).\]
Here $k_1,k_2,k_3\not =0$ in the summation,
\begin{equation*}
\begin{aligned}
&\widetilde{f}_1(\tau,k)=\langle k\rangle^{s^*} \langle\tau-k^3 \rangle^{1/2-\varepsilon}|\widehat{w}(\tau,k)|,\\
&\widetilde{f}_2(\tau,k)=\langle k\rangle^{-s} \langle\tau-k^3 \rangle^{1/2-\varepsilon}|\widehat{\varphi}(\tau,k)|,\\
&\widetilde{f}_3(\tau,k)=\langle k\rangle^{s+2\varepsilon} \langle\tau-k^3 \rangle^{1/2}|\widehat{u}(\tau,k)|,
\end{aligned}
\end{equation*}
and
\[\widetilde{M}=\frac{\langle k_2\rangle^{1+s}}{\langle k_1\rangle^{s^*} \langle k_3\rangle^{s+2\varepsilon}\langle\tau_1-k_1^3 \rangle^{1/2-\varepsilon}\langle\tau_2-k_2^3 \rangle^{1/2-\varepsilon}\langle\tau_3-k_3^3 \rangle^{1/2}}.\]
Let $\widetilde{I}_i$ stand for the contribution in the case of $\max\{\langle\tau_j-k_j^3\rangle\}=\langle\tau_i-k_i^3\rangle$.
The estimates for $\widetilde{I}_1$ is similar to $I_3$ above, namely for $\varepsilon\le (s^*-s)/2$, we have
\[\widetilde{M}\lesssim \frac{\langle k_2\rangle^{1/2+s+\varepsilon}}{\langle k_1\rangle^{1/2+s^*-\varepsilon} \langle k_3\rangle^{1/2+s+\varepsilon}\langle\tau_2-k_2^3 \rangle^{1/2-\varepsilon}\langle\tau_3-k_3^3 \rangle^{1/2}}\lesssim \frac{1}{\langle\tau_2-k_2^3 \rangle^{1/2-\varepsilon}\langle\tau_3-k_3^3 \rangle^{1/2}},\]
which leads to
\[\widetilde{I}_1\lesssim \|w\|_{\dot{X}^{s^*,1/2-\varepsilon}} \|\varphi\|_{\dot{X}^{-s,1/3}} \|u\|_{\dot{X}^{s+2\varepsilon,1/3}}\lesssim \|w\|_{\dot{X}^{s^*,1/2-\varepsilon}} \|\varphi\|_{\dot{X}^{-s,1/2-\varepsilon}}.\]
The estimate for $\widetilde{I}_2$ is analogous. And the estimate for $I_3$ is simpler, since in this case we have
\[\widetilde{M}\lesssim \frac{1}{\langle \tau_1-k_1^3\rangle^{1/2-\varepsilon} \langle \tau_2-k_2^3\rangle^{1/2-\varepsilon}},\]
which leads to
\[\widetilde{I}_3\lesssim \|w\|_{\dot{X}^{s^*,1/3}}\|\varphi\|_{\dot{X}^{-s,1/3}} \|u\|_{\dot{X}^{s+2\varepsilon,1/2}}\lesssim \|w\|_{\dot{X}^{s^*,1/2-\varepsilon}}\|\varphi\|_{\dot{X}^{-s,1/2-\varepsilon}}.\]
Now the proof is complete.
\end{proof}

\subsection{Propagation properties}\label{Appendix-Propagation}

In this subsection, let us recall two propagation results of linear KdV equation, which are involved in the proof of full observability inequality (Lemma~\ref{Lemma-fullobs}). The first is the propagation of compactness.

\begin{proposition}[{\cite[Proposition 3.5]{LRZ-10}}]\label{Prop-propagation-1}
Let $T>0$, $0\leq b'\leq b\leq 1$ with $b>0$, and a non-empty subset $\omega\subset \T$ be arbitrarily given. Assume that the sequences $\{u^n\}\subset X_T^{0,b} $ and $\{f^n\}\subset X_T^{-2+2b,-b}$ satisfy (in the sense of distributions)
\[\partial_t u^n+\partial_x^3u^n=f^n,\quad \|u^n\|_{X_T^{0,b}}\leq C\]
for some constant $C>0$, and 
\[\|u^n\|_{X_T^{-2+2b,-b}}+\|f^n\|_{X_T^{-2+2b,-b}}+\|u^n\|_{X_T^{-1+2b',-b'}}\rightarrow 0.\]
If $u^n\rightarrow 0$ strongly in 
$L^2(0,T;L^2(\omega)),$
then $u^n\rightarrow 0$ strongly in $L^2_{loc}(0,T;L^2(\T)).$
\end{proposition}

And the second is the propagation of regularity.

\begin{proposition}[{\cite[Proposition 3.6]{LRZ-10}}]\label{Prop-propagation-2}
    Let $T>0$, $0\leq b<1$, $r\in\R$, and a non-empty subset $\omega\subset \T$ be arbitrarily given. Assume that $u\in  X^{r,b}_T$ and $f\in X^{r,-b}_T$ satisfy (in the sense of distributions)
    \[\partial_t u+\partial_x^3u=f.\]
    If $u\in L^2_{loc}(0,T;H^{r+\rho}(\omega))$ and $0<\rho\leq \min\{1-b,1/2\}$, then $u\in L^2_{loc}(0,T;H^{r+\rho}(\T))$.
\end{proposition}

\subsection{Hilbert uniqueness method}\label{Section-HUM}

We carry out the derivation from the truncated observability (Lemma~\ref{Lemma-truncatedobs}) to low-frequency controllability (Proposition~\ref{Prop-LFcontrol}). The argument is similar to \cite[Proposition 5.5]{LWX-24} and \cite[Proposition 4.9]{CXZZ-25}, which do not depend on the specific PDE setting. For this reason, we only sketch the proof.

For $s\in [0,1/4+\sigma)$ and $v_0\in \dot{H}^s(\T)$, we introduce a functional $J\colon \dot{H}_m\to \R$ by
\begin{equation}\label{Functional-J}
J(\varphi_1)=\frac{1}{2}\int_{t_1}^{t_2}\|\mathcal{P}_NG\varphi\|^2_{\dot{H}^{-s}(s_1,s_2)}dt+(v_0,\varphi(0))_{\dot{L}^2},\quad \varphi_1\in \dot{H}_m.
\end{equation}
Here $G$ is defined via \eqref{G-definition}, and $\varphi=\mathcal{U}_w (\varphi_1)$ is the solution of adjoint system \eqref{Adjoint-problem}. The following result turns out to be a refinement of Proposition~\ref{Prop-LFcontrol}.

\begin{proposition}\label{Prop-LFcontrol-HUM}
Given $\sigma\in (0,1/4)$, $s\in [0,1/4+\sigma)$, $R>0$ and $m\in\N^+$, there exists a constant $N\in\N^+$ such that the following assertions hold.
\begin{enumerate}
\item[$(1)$] For every $w\in B_{\dot{Y}^{1/4+\sigma}_1}(R)$ and $v_0\in \dot{H}^s$, the functional $J\colon \dot{H}_m\to \R$ defined by \eqref{Functional-J} admits a unique global minimizer $\breve\varphi_1\in \dot{H}_m$. Moreover, by viewing $\dot{H}_m$ as a subspace of $\dot{H}^{-s}$, the HUM-type map defined by
\[\Xi(w)\colon \dot{H}^s\to \dot{H}^{-s},\quad \Xi(w)(v_0):=\breve{\varphi}_1\]
is a bounded linear operator, and the map 
\[B_{\dot{Y}^{1/4+\sigma}_1}(R)\ni w\mapsto\Xi(w)\in\mathcal L(\dot{H}^s,\dot{H}^{-s})\]
is Lipschitz and continuously differentiable.
		
\item[$(2)$] If the control $\breve{\xi}$ is chosen as
\[\breve\xi=(-\Delta_{loc})^{-s}(\mathcal{P}_NG\breve{\varphi}),\quad \breve{\varphi}=\mathcal{U}_w(\breve{\varphi}_1),\]
then
\[\int_{t_1}^{t_2}\|\breve\xi(t)\|^2_{\dot{H}^s(s_1,s_2)}dt\leq C(\sigma,s,R)\|v_0\|^2_{\dot{H}^s},\]
and the solution $\breve v=\mathcal V_w(v_0,\chi \mathcal{P}_N \breve\xi)$ of linearized equation \eqref{Control-problem} satisfies 
\begin{equation*}
P_m{\breve v(1)}=0,
\end{equation*}
\end{enumerate}
\end{proposition}

Assuming its validity for the moment, let us first fulfill the low-frequency controllability.

\begin{proof}[\bf Proof of Proposition~\ref{Prop-LFcontrol}]
Assertion (1) follows immediately from Proposition~\ref{Prop-LFcontrol-HUM}(2), and (2) is due to the fact that the control $\breve{\xi}$ (as a linear operator of $v_0$) is the composition of three bounded linear operators: $(-\Delta_{loc})^{-s}\mathcal{P}_N G$, $\mathcal{U}_w$ and $\Xi(w)$. The first is independent of $w$, and the last two are both Lipschitz and continuously differentiable in $w$, thanks to Proposition~\ref{Prop-linear-2} and Proposition~\ref{Prop-LFcontrol-HUM}(1). The proof is now complete.
\end{proof}

And then we provide the HUM arguments.

\begin{proof}[\bf Proof of Proposition~\ref{Prop-LFcontrol-HUM}]
Let $N\in\N^+$ (depending on $\sigma,s,R$ and $m$) be established in Lemma~\ref{Lemma-truncatedobs}, and fix any $w\in B_{\dot{Y}^{1/4+\sigma}_1}(R)$ and $v_0\in \dot{H}^s$. To prove (1), the existence and uniqueness of global minimizer for $J$ can be derived from strict convexity and coercivity of $J$ (the latter is a consequence of truncated observability \eqref{Truncated-obs}). The linearity of $\Xi(w)$, as well as its Lipschitz and $C^1$ dependence on $w$, are standard facts for quadratic variation problems, whose proof is based on \eqref{Gateaux-derivative} below and implicit function theorem. We refer the careful reader to \cite[Proposition 4.9(3)]{CXZZ-25} for further details.

To prove (2), note that the G\^{a}teaux derivative of $J$ at $\breve{\varphi}_1$ vanishes, i.e.
\begin{align}\label{Gateaux-derivative}
0=\lim_{\varepsilon\rightarrow 0}\frac{J(\breve\varphi_1+\varepsilon\varphi_1)-J(\breve\varphi_1)}{\varepsilon}=\int_{t_1}^{t_2}(\breve\xi,\mathcal{P}_NG\varphi)_{\dot{L}^2(s_1,s_2)}dt+(v_0,\varphi(0))_{\dot{L}^2},\quad \forall \varphi_1\in \dot{H}_m.
\end{align}
Taking $\xi=\breve\xi$ in \eqref{Dual-identity}, we thus obtain
$
(\breve v(1),\varphi_1)=0
$. Since $\varphi_1\in\dot{H}_m$ is arbitrary, we conclude that $P_m\breve v(1)=0$.
Furthermore, taking $\varphi_1=\breve\varphi_1$ in \eqref{Dual-identity} and using \eqref{Truncated-obs}, one can infer that 
\[\int_{t_1}^{t_2}\|\breve\xi(t)\|^2_{\dot{H}^{s}(s_1,s_2)}dt=-(v_0,\breve\varphi(0))_{\dot{L}^2}\leq C\|v_0\|_{\dot{H}^s}\|\breve\varphi_1\|_{\dot{H}^{-s}}\leq C\|v_0\|_{\dot{H}^s}^2+\frac{1}{2}\int_{t_1}^{t_2}\|\breve\xi(t)\|^2_{\dot{H}^{s}(s_1,s_2)}dt.\]
Absorb the last term to the left-hand side concludes the proof of (2).
\end{proof}

\normalem
\bibliographystyle{plain}
\bibliography{References}
	
\end{document}